\documentclass[12pt,reqno]{amsart}
\usepackage{amsmath,amssymb,amsfonts,amscd,latexsym,amsthm,mathrsfs,verbatim,comment,cite}
\usepackage{hyperref}
\newtheorem{theorem}{Theorem}[section]
\newtheorem{conjecture}[theorem]{Conjecture}
\newtheorem{lemma}[theorem]{Lemma}
\let\wt\widetilde

\renewcommand{\d}{{\mathrm d}}
\newcommand{\qbin}[2]{\genfrac{[}{]}{0pt}{}{#1}{#2}}
\begin{document}

\title{A $q$-microscope for supercongruences}

\date{11 March 2018. \emph{Revised}: 13 February 2019}

\author{Victor J. W. Guo}
\address{School of Mathematical Sciences, Huaiyin Normal University, Huai'an 223300, Jiangsu, People's Republic of China}
\email{jwguo@hytc.edu.cn}

\author{Wadim Zudilin}
\address{Department of Mathematics, IMAPP, Radboud University, PO Box 9010, 6500~GL Nijmegen, Netherlands}
\email{w.zudilin@math.ru.nl}

\address{School of Mathematical and Physical Sciences, The University of Newcastle, Callag\-han, NSW 2308, Australia}
\email{wadim.zudilin@newcastle.edu.au}

\thanks{The first author was partially supported by the National Natural Science Foundation of China (grant 11771175).}

\subjclass[2010]{11B65, 11F33, 11Y60, 33C20, 33D15}
\keywords{$1/\pi$; Ramanujan; $q$-analogue; cyclotomic polynomial; radial asymptotics; WZ pair; basic hypergeometric function; (super)congruence.}

\begin{abstract}
By examining asymptotic behavior of certain infinite basic ($q$-) hypergeometric sums at roots of unity
(that is, at a `$q$-microscopic' level) we prove polynomial congruences for their truncations.
The latter reduce to non-trivial (super)congruences for truncated ordinary hypergeometric sums,
which have been observed numerically and proven rarely. A typical example includes derivation, from a $q$-analogue of
Ramanujan's formula
$$
\sum_{n=0}^\infty\frac{\binom{4n}{2n}{\binom{2n}{n}}^2}{2^{8n}3^{2n}}\,(8n+1)
=\frac{2\sqrt{3}}{\pi},
$$
of the two supercongruences
$$
S(p-1)\equiv p\biggl(\frac{-3}p\biggr)\pmod{p^3}
\quad\text{and}\quad
S\Bigl(\frac{p-1}2\Bigr)
\equiv p\biggl(\frac{-3}p\biggr)\pmod{p^3},
$$
valid for all primes $p>3$,
where $S(N)$ denotes the truncation of the infinite sum at the $N$-th place and $\bigl(\frac{-3}{\cdot}\bigr)$ stands for the quadratic character modulo~$3$.
\end{abstract}

\maketitle

\section{Introduction}
\label{s1}

In our study, through several years, of Ramanujan's and Ramanujan-type
formulae \cite{Ra14} for $1/\pi$, a lot of arithmetic mystery have been discovered along the way.
A typical example on the list is the identity
\begin{equation}
\sum_{n=0}^\infty\frac{\binom{4n}{2n}{\binom{2n}{n}}^2}{2^{8n}3^{2n}}\,(8n+1)
=\frac{2\sqrt{3}}{\pi}.
\label{ram1}
\end{equation}
Part of the arithmetic story, which is the main topic of the present note,
is a production of Ramanujan-type supercongruences \cite{Zu09} (with some particular instances indicated in the earlier work \cite{Hamme} of Van Hamme):
truncation of a Ramanujan-type infinite sum at the $(p-1)$-th place happens to be a simple expression modulo $p^3$ for all but finitely many primes $p$.
In our example \eqref{ram1}, the result reads
\begin{equation}
\sum_{k=0}^{p-1}\frac{\binom{4k}{2k}{\binom{2k}{k}}^2}{2^{8k}3^{2k}}\,(8k+1)
\equiv p\biggl(\frac{-3}p\biggr)\pmod{p^3}
\quad\text{for $p>3$ prime},
\label{ram1a}
\end{equation}
where the Jacobi--Kronecker symbol $\bigl(\frac{-3}{\cdot}\bigr)$ `replaces' the square root of $3$. Another experimental observation,
which seems to be true in several cases but not in general, is that truncation of the sum at the $(p-1)/2$-th place, results
in a similar congruence with the same right-hand side, like
\begin{equation}
\sum_{k=0}^{(p-1)/2}\frac{\binom{4k}{2k}{\binom{2k}{k}}^2}{2^{8k}3^{2k}}\,(8k+1)
\equiv p\biggl(\frac{-3}p\biggr)\pmod{p^3}
\quad\text{for $p>3$ prime},
\label{ram1b}
\end{equation}
in the example above.
By noticing that the intermediate terms corresponding to $k$ in the range $(p-1)/2<k\le p-1$ are not
necessarily $0$ modulo $p^3$, we conclude that \eqref{ram1a} and \eqref{ram1b} are in fact different congruences.
The experimental findings \eqref{ram1a}, \eqref{ram1b} were implicitly discussed in \cite{Zu09}
and later recorded in \cite[Conjecture 5.6]{Sun11}.

Development of methods \cite{GuZu,He,Kilbourn,Long,MO08,Mortenson,OZ,Swisher,Hamme,Zu09} for establishing
Ramanu\-jan-type supercongruences like \eqref{ram1a} and
\eqref{ram1b}, sometimes modulo a smaller power of $p$ and normally on
a case-by-case study, was mainly hypergeometric. It involved tricky
applications of numerous hypergeometric identities and use of the
algorithm of creative telescoping, namely of suitable WZ
(Wilf--Zeilberger) pairs. These strategies have finally led \cite{GL18,GZ18} to
$q$-analogues of Ramanujan-type formulae for $1/\pi$ including
\begin{equation}
\sum_{n=0}^{\infty}\frac{q^{2n^2}(q;q^2)_n^2 (q;q^2)_{2n}}{(q^2;q^2)_{2n}(q^6;q^6)_n^2 }[8n+1]
=\frac{(q^3;q^2)_\infty (q^3;q^6)_\infty }{(q^2;q^2)_\infty (q^6;q^6)_\infty}.
\label{q4}
\end{equation}
At this stage we already need to familiarize ourselves with standard hypergeometric notation.
We always consider $q$ inside the unit disc, $|q|<1$, and define
$$
(a;q)_\infty=\prod_{j=0}^\infty(1-aq^j).
$$
Then the $q$-Pochhammer symbol and its non-$q$-version are given by
$$
(a;q)_n=\frac{(a;q)_\infty}{(aq^n;q)_\infty}=\prod_{j=0}^{n-1}(1-aq^j)
\quad\text{and}\quad
(a)_n=\frac{\Gamma(a+n)}{\Gamma(a)}=\prod_{j=0}^{n-1}(a+j)
$$
for non-negative integers $n$, so that
$$
\lim_{q\to1}\frac{(q^a;q)_n}{(1-q)^n}=(a)_n
\quad\text{and}\quad
\lim_{q\to1}\frac{(q;q)_\infty(1-q)^{1-a}}{(q^a;q)_\infty}=\Gamma(a).
$$
The related $q$-notation also includes the $q$-numbers and $q$-binomial coefficients
$$
[n]=[n]_q=\frac{1-q^n}{1-q}
\quad\text{and}\quad
\qbin nm={\qbin nm}_q=\frac{(q;q)_n}{(q;q)_m(q;q)_{n-m}}.
$$
In the case of formula~\eqref{q4}, we see that
\begin{gather*}
\lim_{q\to1}\frac{(q;q^2)_{2n}}{(q^2;q^2)_{2n}}
=\frac{(\frac12)_{2n}}{(2n)!}=\frac{(\frac14)_n(\frac34)_n}{(\frac12)_n\,n!},
\\
\lim_{q\to1}\frac{(q;q^2)_n}{(q^6;q^6)_n}
=\lim_{q\to1}\frac{(q;q^2)_n}{(q^2;q^2)_n\,\prod_{j=1}^n(1+q^{2j}+q^{4j})}
=\frac{(\frac12)_n}{n!\,3^n}
\\ \intertext{and}
\lim_{q\to1}\frac{(q^3;q^6)_\infty}{(q^6;q^6)_\infty(1-q^6)^{1/2}}
=\lim_{q\to1}\frac{(q;q^2)_\infty}{(q^2;q^2)_\infty(1-q^2)^{1/2}}
=\frac{1}{\Gamma(\tfrac12)}=\frac{1}{\sqrt\pi},
\end{gather*}
hence in the limit as $q\to1$ we obtain
$$
\sum_{n=0}^\infty\frac{(\frac14)_n(\frac12)_n(\frac34)_n}{n!^3\,9^n}\,(8n+1)
=\frac{2\sqrt{3}}{\pi}.
$$
This equality transforms into \eqref{ram1} after a simple manipulation of the Pochhammer symbols.

What are $q$-analogues good for?

It is not hard to imagine that suitable truncations of $q$-sums like \eqref{q4} satisfy certain $q$-analogues of supercongruences of type \eqref{ram1a} or \eqref{ram1b}.
It is also reasonable to expect that the earlier strategies
for establishing Ramanujan-type supercongruences possess suitable $q$-analogues. This is indeed the case, and in a series
of papers \cite{Guo1,Guo2,Guo3,Guo4,Guo2018} the first author uses the $q$-WZ machinery to produce many such examples of $q$-supercongruences,
in particular, $q$-analogues for those from Van Hamme's famous list~\cite{Hamme}. Drawbacks of this approach are
short supply of wanted $q$-WZ pairs and lack of $q$-analogues of classical congruences in the required strength.

In this note we offer a different strategy for proving $q$-congruences. The idea rests on the fact that the asymptotic behavior
of an infinite $q$-sum at roots of unity is determined by its truncation evaluated at the roots.
This leads us to a natural extraction of the truncated sum and its evaluation modulo cyclotomic polynomials
$$
\Phi_n(q)=\prod_{\substack{j=1\\ \gcd(j,n)=1}}^n(q-e^{2\pi ij/n})\in\mathbb Z[q]
$$
and their products. In what follows, the congruence $A_1(q)/A_2(q)\equiv0\pmod{P(q)}$
for polynomials $A_1(q),A_2(q),P(q)\in\mathbb Z[q]$ is understood as $P(q)$ divides $A_1(q)$ and is coprime with $A_2(q)$;
more generally, $A(q)\equiv B(q)\pmod{P(q)}$ for rational functions $A(q),B(q)\in\mathbb Z(q)$ means $A(q)-B(q)\equiv0\pmod{P(q)}$.

Our principal result in this direction is the following theorem observed experimentally in \cite{GZ18}.

\begin{theorem}
\label{thm:8k+1}
Let $n$ be a positive integer coprime with $6$. Then
\begin{align}
\sum_{k=0}^{n-1}\frac{(q;q^2)_k^2 (q;q^2)_{2k}}{(q^2;q^2)_{2k}(q^6;q^6)_k^2}[8k+1]q^{2k^2}
&\equiv q^{-(n-1)/2}[n]\biggl(\frac{-3}{n}\biggr) \pmod{[n]\Phi_n(q)^2},
\label{q4a}
\\
\sum_{k=0}^{(n-1)/2}\frac{(q;q^2)_k^2 (q;q^2)_{2k}}{(q^2;q^2)_{2k}(q^6;q^6)_k^2}[8k+1]q^{2k^2}
&\equiv q^{-(n-1)/2}[n]\biggl(\frac{-3}{n}\biggr) \pmod{[n]\Phi_n(q)^2}.
\label{q4b}
\end{align}
\end{theorem}

Clearly, the limiting case $q\to1$ for $n=p$ leads to the Ramanujan-type supercongruences \eqref{ram1a} and \eqref{ram1b};
importantly, it also leads to more general supercongruences when choosing $n=p^s$, an arbitrary power of prime $p>3$.
The significance of our proof is that it really deals with the $q$-hypergeometric sum \eqref{q4} at a `$q$-microscopic' level
(that is, at roots of unity), hence it cannot be transformed into a derivation of \eqref{ram1a} and \eqref{ram1b} directly from~\eqref{ram1}.

Our proof of Theorem~\ref{thm:8k+1} combines two principles. One corresponds to achieving the congruences in \eqref{q4a} and \eqref{q4b}
modulo $[n]$ only, and it can be easier illustrated in the following `baby' situations also from \cite{GZ18}.

\begin{theorem}
\label{th:baby1}
Let $n$ be a positive odd integer. Then
\begin{align}
\sum_{k=0}^{n-1}(-1)^k \frac{(q;q^2)_k
(-q;q^2)_k^2}{(q^4;q^4)_k(-q^4;q^4)_k^2}[6k+1]q^{3k^2} &\equiv
0\pmod{[n]}, \label{eq:33}
\\
\sum_{k=0}^{(n-1)/2}(-1)^k\frac{(q;q^2)_k
(-q;q^2)_k^2}{(q^4;q^4)_k(-q^4;q^4)_k^2}[6k+1]q^{3k^2} &\equiv
0\pmod{[n]}. \label{eq:44}
\end{align}
\end{theorem}

\begin{theorem}
\label{th:baby2}
Let $n$ be a positive odd integer. Then
\begin{align}
\sum_{k=0}^{n-1}\frac{(q^2;q^4)_k
(-q;q^2)_k^2}{(q^4;q^4)_k(-q^4;q^4)_k^2}[6k+1]q^{k^2} &\equiv
0\pmod{[n]}, \label{eq:11}
\\
\sum_{k=0}^{(n-1)/2}\frac{(q^2;q^4)_k
(-q;q^2)_k^2}{(q^4;q^4)_k(-q^4;q^4)_k^2}[6k+1]q^{k^2} &\equiv
0\pmod{[n]}. \label{eq:22}
\end{align}
\end{theorem}

The second principle is about getting one more parameter involved in the $q$-story,
to opt for `creative microscoping'.

\begin{theorem}\label{thm:8k+1-a}
Let $n$ be a positive integer coprime with $6$. Then, for any indeterminates $a$ and $q$, we have
modulo $[n](1-aq^n)(a-q^n)$,
\begin{align}
\sum_{k=0}^{n-1}\frac{(aq;q^2)_k (q/a;q^2)_k (q;q^2)_{2k}}{(q^2;q^2)_{2k}(aq^6;q^6)_k (q^6/a;q^6)_k }[8k+1]q^{2k^2}
&\equiv q^{-(n-1)/2}[n]\biggl(\frac{-3}{n}\biggr),
\label{q4a-new}
\\
\sum_{k=0}^{(n-1)/2}\frac{(aq;q^2)_k (q/a;q^2)_k (q;q^2)_{2k}}{(q^2;q^2)_{2k}(aq^6;q^6)_k (q^6/a;q^6)_k }[8k+1]q^{2k^2}
&\equiv q^{-(n-1)/2}[n]\biggl(\frac{-3}{n}\biggr).
\label{q4b-new}
\end{align}
\end{theorem}

Our exposition below is as follows. In Section~\ref{s2} we prove Theorems~\ref{th:baby1}, \ref{th:baby2} and highlight some
difficulties in doing so by the $q$-WZ method. In Section~\ref{s3} we demonstrate Theorem~\ref{thm:8k+1-a}
and show how it implies Theorem~\ref{thm:8k+1}. Section~\ref{s4} contains several further results on
Ramanujan-type $q$-supercongruences and outlines of their proofs. Known congruences and $q$-congruences
for truncated hypergeometric sums already form a broad area of research; in our final Section~\ref{s5} we
record several open problems and directions which will initiate further development of the method in this note
and of traditional hypergeometric techniques.

We remark that asymptotic behavior of $q$-series at roots of unity attracts a lot of attention in recent studies
of mock theta functions and so-called quantum modular forms; we limit our citations about related notion and results
to \cite{FOR13,Za10}. This gives a good indication, at least a hope, that the methods developed in those areas may shed
some light on $q$-supercongruences and their $q\to1$ implications.

\section{Asymptotics at roots of unity}
\label{s2}

A fundamental principle for computing basic hypergeometric sums at roots of unity is encoded in the following
simple observation known as the $q$-Lucas theorem (see \cite{Olive} and \cite[Proposition 2.2]{De}).

\begin{lemma}
\label{prop:root}
Let $\zeta$ be a primitive $d$-th root of unity and let $a,b,\ell,k$ be non-negative integers with $0\le b,k\le d-1$.
Then
$$
{\qbin{ad+b}{\ell d+k}}_\zeta=\binom a\ell{\qbin bk}_\zeta.
$$
\end{lemma}

We recall that the congruences \eqref{eq:33}--\eqref{eq:22} are motivated by the following $q$-hyper\-geo\-metric identities:
\begin{align}
\sum_{k=0}^\infty(-1)^k\frac{(q;q^2)_k
(-q;q^2)_k^2}{(q^4;q^4)_k(-q^4;q^4)_k^2}[6k+1]q^{3k^2}
&=\frac{(q^3;q^4)_\infty (q^5;q^4)_\infty}{(-q^4;q^4)_\infty^2},
\label{eq:inf-2}
\\
\sum_{k=0}^\infty\frac{(q^2;q^4)_k
(-q;q^2)_k^2}{(q^4;q^4)_k(-q^4;q^4)_k^2}[6k+1]q^{k^2}
&=\frac{(-q^2;q^4)_\infty^2}{(1-q)(-q^4;q^4)_\infty^2},
\label{eq:inf-1}
\end{align}
derived in \cite{GZ18} with a help of the quadratic transformation \cite[eq.~(4.6)]{Ra93}.
The equalities \eqref{eq:inf-2}, \eqref{eq:inf-1} can be written as
\begin{align}
\sum_{k=0}^\infty(-1)^k{\qbin{2k}{k}}_q\frac{(-q;q^2)_k^2 [6k+1]q^{3k^2}}{(-q;q)_k^2(-q^2;q^2)_k(-q^4;q^4)_k^2}
&=\frac{(q^3;q^2)_\infty}{(-q^4;q^4)_\infty^2},
\label{eq:inf-4}
\\
\sum_{k=0}^\infty{\qbin{2k}{k}}_{q^2}\frac{(-q;q^2)_k^2 [6k+1]q^{k^2}}{(-q^2;q^2)_k^2 (-q^4;q^4)_k^2}
&=\frac{(-q^2;q^4)_\infty^2}{(1-q)(-q^4;q^4)_\infty^2}.
\label{eq:inf-3}
\end{align}

\begin{proof}[Proof of Theorem~\textup{\ref{th:baby1}}]
It is immediate that \eqref{eq:33} and \eqref{eq:44} are true for $n=1$.

For $n>1$, let $\zeta\ne1$ be an $n$-th root of unity, not necessarily primitive.
This means that $\zeta$ is a primitive root of unity of odd degree $d\mid n$.
For the right-hand side in \eqref{eq:inf-4}, we clearly have
\begin{equation}
\frac{(q^3;q^2)_\infty}{(-q^4;q^4)_\infty^2}=\prod_{j=1}^\infty\frac{1-q^{2j+1}}{(1+q^{4j})^2}\to0
\quad\text{as}\; q\to\zeta,
\label{zzeta1}
\end{equation}
because the numerator of the product `hits' the zero at $q=\zeta$, while the denominator does not vanish at $q=\zeta$.
If $c_q(k)$ denotes the $k$-th term on the left-hand side in \eqref{eq:inf-4} (or \eqref{eq:inf-2}),
then we write the left-hand side as
$$
\sum_{\ell=0}^\infty c_q(\ell d)\sum_{k=0}^{d-1}\frac{c_q(\ell d+k)}{c_q(\ell d)}.
$$
Observe that, for the internal terms,
$$
\lim_{q\to\zeta}\frac{c_q(\ell d+k)}{c_q(\ell d)}
=c_\zeta(k),
$$
and also that
\begin{equation}
c_\zeta(k)=0 \quad\text{for $k$ in the range}\; (d-1)/2<k\le d-1,
\label{zzeta2}
\end{equation}
because of the factor
$(q;q^2)_k=\prod_{j=1}^k(1-q^{2j-1})$ in the numerator of $c_q(k)$ (as seen in \eqref{eq:inf-2}).
With the help of Lemma~\ref{prop:root},
$$
\lim_{q\to\zeta}c_q(\ell d)=c_\zeta(\ell d)
=(-1)^\ell\binom{2\ell}{\ell}\frac{(-\zeta;\zeta^2)_d^{2\ell}}
{(-\zeta;\zeta)_d^{2\ell}(-\zeta^2;\zeta^2)_d^{\ell}(-\zeta^4;\zeta^4)_d^{2\ell}}
=\frac{(-1)^\ell}{8^\ell}\binom{2\ell}{\ell},
$$
since $\zeta,\zeta^2,\zeta^4$ are all primitive $d$-th roots of unity, and
\begin{equation}
\begin{gathered}
(-\zeta;\zeta)_{d}=(1+\zeta^d)\prod_{j=1}^{d-1}(1+\zeta^j)=2,
\quad
(-\zeta^2;\zeta^2)_{d}=(-\zeta^4;\zeta^4)_{d}=2,
\\
(-\zeta;\zeta^2)_{d}
=(1+\zeta)(1+\zeta^3)\dotsb(1+\zeta^d)(1+\zeta^{d+2})\dotsb(1+\zeta^{2d-1})=(-\zeta;\zeta)_{d}=2.
\end{gathered}
\label{eq:zeta-sim}
\end{equation}
Using the binomial theorem
$$
(1-z)^{-1/2}=\sum_{\ell=0}^\infty\frac{(\frac12)_\ell}{\ell!}\,z^\ell
=\sum_{\ell=0}^\infty\binom{2\ell}{\ell}\biggl(\frac z4\biggr)^\ell
$$
we deduce that
\begin{equation}
\sum_{\ell=0}^\infty (-1)^\ell\binom{2\ell}{\ell}\frac{1}{8^\ell}=\frac{\sqrt{6}}{3}.
\label{eq:xi-2}
\end{equation}
It follows that the limiting case of the equality \eqref{eq:inf-4} as $q\to\zeta$ assumes the form
$$
\frac{\sqrt{6}}{3}\sum_{k=0}^{d-1}c_\zeta(k)
=\frac{\sqrt{6}}{3}\sum_{k=0}^{(d-1)/2}c_\zeta(k)
=0,
$$
where the first equality follows from \eqref{zzeta2} and the second equality is implied by~\eqref{zzeta1}.
By above it remains to notice that
$$
\sum_{k=0}^{n-1}c_\zeta(k)=\sum_{\ell=0}^{n/d-1}\sum_{k=0}^{d-1}c_\zeta(\ell d+k)
=\sum_{\ell=0}^{n/d-1}\frac{(-1)^\ell }{8^\ell}
{2\ell \choose \ell }\sum_{k=0}^{d-1}c_\zeta(k)=0
$$
and
$$
\sum_{k=0}^{(n-1)/2}c_\zeta(k)
=\sum_{\ell=0}^{(n/d-3)/2}\frac{(-1)^\ell }{8^\ell}{2\ell \choose \ell }\sum_{k=0}^{d-1}c_\zeta(k)+\sum_{k=0}^{(d-1)/2}c_\zeta((n-d)/2+k)=0,
$$
which imply that both sums $\sum_{k=0}^{n-1}c_q(k)$ and $\sum_{k=0}^{(n-1)/2}c_q(k)$ are divisible
by the cyclotomic polynomial $\Phi_d(q)$. Since this is true for any divisor $d>1$ of $n$, we conclude that they
are divisible by
\begin{equation*}
\prod_{\substack{d\mid n\\d>1}}\Phi_d(q)=[n].
\qedhere
\end{equation*}
\end{proof}

\begin{proof}[Proof of Theorem~\textup{\ref{th:baby2}}]
Similarly, we can prove \eqref{eq:11} and \eqref{eq:22}. The difference is that we replace the evaluation \eqref{eq:xi-2} with
\begin{align*}
\sum_{\ell=0}^\infty\binom{2\ell}{\ell}\frac{(-\zeta;\zeta^2)_{d}^{2\ell}}{(-\zeta^2;\zeta^2)_{d}^{2\ell}(-\zeta^4;\zeta^4)_{d}^{2\ell}}
=\sum_{\ell=0}^\infty\binom{2\ell}{\ell}\frac{1}{4^\ell}=\infty,
\end{align*}
while the right-hand side in \eqref{eq:inf-3} is uniformly bounded as $q\to\zeta$. Indeed, the latter follows from
\begin{align*}
\lim_{q\to\zeta}\frac{(-q^2;q^4)_{\ell d+k}^2}{(1-q)(-q^4;q^4)_{\ell d+k}^2}
=\frac{(-\zeta^2;\zeta^4)_{\ell d+k}^2}{(1-\zeta)(-\zeta^4;\zeta^4)_{\ell d+k}^2}
=\frac{(-\zeta^2;\zeta^4)_k^2}{(1-\zeta)(-\zeta^4;\zeta^4)_k^2}
\end{align*}
for any $\ell\ge0$ and $0\le k<d$ (see~\eqref{eq:zeta-sim}), so that the expression
$$
\frac{(-q^2;q^4)_n^2}{(1-q)(-q^4;q^4)_n^2}
$$
is bounded above by
$$
\frac1{|1-\zeta|}\max_{0\le k<d}\frac{|(-\zeta^2;\zeta^4)_k|^2}{|(-\zeta^4;\zeta^4)_k|^2}+1
$$
as $q\to\zeta$. Thus, we conclude that
\begin{align*}
\sum_{k=0}^{d-1}{\qbin{2k}{k}}_{\zeta^2}\frac{(-\zeta;\zeta^2)_k^2 [6k+1]_\zeta\zeta^{k^2}}{(-\zeta^2;\zeta^2)_k^2(-\zeta^4;\zeta^4)_k^2}
=\sum_{k=0}^{(d-1)/2}{\qbin{2k}{k}}_{\zeta^2}\frac{(-\zeta;\zeta^2)_k^2 [6k+1]_\zeta\zeta^{k^2}}{(-\zeta^2;\zeta^2)_k^2(-\zeta^4;\zeta^4)_k^2}
=0
\end{align*}
for any (odd) divisor $d>1$ of $n$, and this leads\,---\,in exactly the same way
as in the proof of Theorem~\ref{th:baby1}\,---\,to the divisibility of the truncated $q$-sums by $[n]$.
\end{proof}

With the summands in Theorems \ref{th:baby1} and \ref{th:baby2}, we can associate the $q$-WZ pairs
\begin{align*}
F(n,k) &=(-1)^{n+k}\frac{[6n-2k+1](q;q^2)_{n-k}(-q;q^2)_{n-k}(-q;q^2)_{n+k}}{(q^4;q^4)_{n} (-q^4;q^4)_{n}(-q^4;q^4)_{n-k}}, \\
G(n,k) &=\frac{(-1)^{n+k}(q;q^2)_{n-k}(-q;q^2)_{n-k}(-q;q^2)_{n+k-1}}{(1-q)(q^4;q^4)_{n-1}(-q^4;q^4)_{n-1}(-q^4;q^4)_{n-k}},
\end{align*}
and
\begin{equation}
\begin{aligned}
\wt F(n,k)
&=\frac{q^{(n-k)^2}[6n-2k+1](q^2;q^4)_{n}(-q;q^2)_{n-k}(-q;q^2)_{n+k}}{(q^4;q^4)_{n}(-q^4;q^4)_{n}(-q^4;q^4)_{n-k}(-q^2;q^4)_k},
\\
\wt G(n,k)
&=\frac{q^{(n-k)^2}(q^2;q^4)_n (-q;q^2)_{n-k}(-q;q^2)_{n+k-1}}{(1-q)(q^4;q^4)_{n-1}(-q^4;q^4)_{n-1} (-q^4;q^4)_{n-k}(-q^2;q^4)_k},
\end{aligned}
\label{second-WZ-pair}
\end{equation}
respectively, where the convention $1/(q^4;q^4)_m=0$ for negative integers $m$ is applied
(and justified by the extended definition
$$
(a;q)_{-n}=\frac{(-1)^na^nq^{n(n+1)/2}}{(q/a;q)_n}
\quad\text{for}\; n=1,2,\dots
$$
of the $q$-Pochhammer symbol).
However, we do not see a way to use the pairs for proving the required congruences.
Let us illustrate the difficulty in the example of the $q$-WZ pair \eqref{second-WZ-pair}, which satisfies
$$
\wt F(n,k-1)-\wt F(n,k)=\wt G(n+1,k)-\wt G(n,k),
$$
and the related congruence \eqref{eq:11}. For an odd integer $m>1$, summing the last equality
over $n=0,1,\dots,(m-1)/2$ we obtain, via telescoping,
$$
\sum_{n=0}^{(m-1)/2}F(n,k-1)-\sum_{n=0}^{(m-1)/2}F(n,k)
=G\Bigl(\frac{m+1}{2},k\Bigr)\equiv0\pmod{[m]}
$$
for any integer $k$, so that
$$
\sum_{n=0}^{(m-1)/2}F(n,0)
\equiv\sum_{n=0}^{(m-1)/2}F(n,1)
\equiv\dots
\equiv\sum_{n=0}^{(m-1)/2}F(n,k)
\pmod{[m]}.
$$
At the same time, there seems to be no natural choice of $k$, for which
$$
\sum_{n=0}^{(m-1)/2}F(n,k)
\equiv0\pmod{[m]}
$$
follows. The same obstacles occur for the remaining congruences in Theorems~\ref{th:baby1} and~\ref{th:baby2}.

\section{$q$-Supercongruences of Ramanujan type}
\label{s3}

We recall that the $q$-analogue \eqref{q4} of Ramanujan's formula \eqref{ram1} is established in \cite{GZ18} on the basis of
\begin{align}
&
\sum_{k=0}^\infty\frac{(1-acq^{4k})(a;q)_k(q/a;q)_k (ac;q)_{2k}}
{(1-ac)(cq^3;q^3)_k(a^2cq^2;q^3)_k(q;q)_{2k}}\,q^{k^2}
\nonumber\\ &\qquad
=\frac{(acq^2;q^3)_\infty(acq^3;q^3)_\infty(aq;q^3)_\infty(q^2/a;q^3)_\infty}
{(q;q^3)_\infty(q^2;q^3)_\infty(a^2cq^2;q^3)_\infty(cq^3;q^3)_\infty}.
\label{cubic-tr}
\end{align}
Replace $q$ with $q^2$, take $c=q/a$ and then $aq$ for $a$:
\begin{align}
&
\sum_{k=0}^\infty[8k+1]\frac{(aq;q^2)_k(q/a;q^2)_k (q;q^2)_{2k}}
{(q^2;q^2)_{2k}(aq^6;q^6)_k(q^6/a;q^6)_k}\,q^{2k^2}
\nonumber\\ &\qquad
=\frac{(q^5;q^6)_\infty(q^7;q^6)_\infty(aq^3;q^6)_\infty(q^3/a;q^6)_\infty}
{(q^2;q^6)_\infty(q^4;q^6)_\infty(aq^6;q^6)_\infty(q^6/a;q^6)_\infty}.
\label{q4t}
\end{align}
Observe that the truncated $q$-sums in Theorem~\ref{thm:8k+1-a} correspond to the left-hand side of \eqref{q4t}.
Furthermore,
\begin{equation}
\frac{(q^5;q^6)_\infty(q^7;q^6)_\infty}{(q^2;q^6)_\infty(q^4;q^6)_\infty}
=\frac{(q;q^2)_\infty(q^6;q^6)_\infty}{(1-q)\,(q^2;q^2)_\infty(q^3;q^6)_\infty}.
\label{eq:omega3}
\end{equation}

\begin{lemma}
\label{lem-new}
Let $n$ be a positive odd integer. Then
\begin{equation}
\sum_{k=0}^{(n-1)/2}\frac{(q^{1-n};q^2)_k (q^{1+n};q^2)_k (q;q^2)_{2k}}{(q^2;q^2)_{2k}(q^{6-n};q^6)_k (q^{6+n};q^6)_k }[8k+1]q^{2k^2}
=q^{-(n-1)/2}[n]\biggl(\frac{-3}{n}\biggr).
\label{eq:lem3.1}
\end{equation}
\end{lemma}

\begin{proof}
We substitute $a=q^{n}$ into \eqref{q4t}.
Then the left-hand side of \eqref{q4t} terminates at $k=(n-1)/2$; therefore, it is exactly
the left-hand side of \eqref{eq:lem3.1}, while the right-hand side of \eqref{q4t} becomes
$q^{-(n-1)/2}[n]$ if $n\equiv1\pmod 3$, $-q^{-(n-1)/2}[n]$ if $n\equiv2\pmod 3$, and 0 if $3\mid n$.
\end{proof}

\begin{lemma}
\label{lem-other}
Let $n$ be an integer with $n>1$ and $(n,6)=1$. Then
\begin{equation*}
\sum_{k=0}^{(n-1)/2}\frac{1-q^{1-n+8k}}{1-q^{1-n}}\,\frac{(aq;q^2)_k(q/a;q^2)_k (q^{1-n};q^2)_{2k}}
{(q^2;q^2)_{2k}(aq^{6-n};q^6)_k(q^{6-n}/a;q^6)_k}\,q^{2k^2}
=0.
\end{equation*}
\end{lemma}

\begin{proof}
We specialise \eqref{cubic-tr} by taking $q^2$ for $q$, then $c=q^{1-n}/a$ and replacing $a$ with $aq$.
The right-hand side of the resulting identity vanishes, because of the factors $(q^{5-n};q^6)_\infty$ and $(q^{7-n};q^6)_\infty$
in the numerator, while the left-hand side terminates at $k=(n-1)/2$ (in fact, even earlier, at $k=\lfloor n/4\rfloor$), since the summand involves $(q^{1-n};q^2)_{2k}$.
\end{proof}

\begin{proof}[Proof of Theorem~\textup{\ref{thm:8k+1-a}}]
Let $\zeta\neq 1$ be a primitive $d$-th root of unity, where $d\mid n$ and $n>1$ is coprime with 6.
Denote by
$$
c_q(k)=[8k+1]\frac{(aq;q^2)_k(q/a;q^2)_k (q;q^2)_{2k}}{(q^2;q^2)_{2k}(aq^6;q^6)_k(q^6/a;q^6)_k}\,q^{2k^2}
$$
the $k$-th term of the sum \eqref{q4t}. Use \eqref{eq:omega3} to write the identity in \eqref{q4t} in the form
\begin{equation}
\sum_{\ell=0}^{\infty}c_q(\ell d)\sum_{k=0}^{d-1}\frac{c_q(\ell d+k)}{c_q(\ell d)}
=\frac{(q;q^2)_\infty(q^6;q^6)_\infty(aq^3;q^6)_\infty(q^3/a;q^6)_\infty}
{(1-q)\,(q^2;q^2)_\infty(q^3;q^6)_\infty(aq^6;q^6)_\infty(q^6/a;q^6)_\infty}.
\label{q4u}
\end{equation}
Consider the limit as $q\to\zeta $ radially.
On the left-hand side we get
$$
\lim_{q\to\zeta}\frac{c_q(\ell d+k)}{c_q(\ell d)}
=\frac{c_\zeta(\ell d+k)}{c_\zeta(\ell d)}
=c_\zeta(k)
$$
and, since $(q;q^2)_{2k}/(q^2;q^2)_{2k}={4k\brack 2k}_{q^2}/(-q;q)_{4k}$, by Lemma \ref{prop:root} and the reduction formulas \eqref{eq:zeta-sim},
$$
\lim_{q\to\zeta}c_q(\ell d)
=\frac1{2^{4\ell}}\binom{4\ell}{2\ell}
\frac{(a\zeta;\zeta^2)_{\ell d}(\zeta/a;\zeta^2)_{\ell d}}{(a\zeta^6;\zeta^6)_{\ell d}(\zeta^6/a;\zeta^6)_{\ell d}}
=\frac1{2^{4\ell}}\binom{4\ell}{2\ell}.
$$
By Stirling's approximation, $\binom{4\ell}{2\ell}\sim2^{4\ell}/\sqrt{2\pi\ell}$ as $\ell\to\infty$, hence
$$
\sum_{\ell=0}^\infty\frac1{2^{4\ell}}\binom{4\ell}{2\ell}=\infty.
$$
For the right-hand side of \eqref{q4u},
\begin{align*}
&\lim_{q\to\zeta}\frac{(q;q^2)_{\ell d+k}(q^6;q^6)_{\ell d+k}(aq^3;q^6)_{\ell d+k}(q^3/a;q^6)_{\ell d+k}}
{(1-q)\,(q^2;q^2)_{\ell d+k}(q^3;q^6)_{\ell d+k}(aq^6;q^6)_{\ell d+k}(q^6/a;q^6)_{\ell d+k}}
\\
&\qquad
=\frac{(\zeta;\zeta^2)_k(\zeta^6;\zeta^6)_k(a\zeta^3;\zeta^6)_k(\zeta^3/a;\zeta^6)_k}
{(1-\zeta)\,(\zeta^2;\zeta^2)_k(\zeta^3;\zeta^6)_k(a\zeta^6;\zeta^6)_k(\zeta^6/a;\zeta^6)_k}
\end{align*}
for any $\ell\ge0$ and $0\le k<d$. Therefore, the expression on the right-hand side of \eqref{q4u}
is bounded above by
$$
\max_{0\le k<d}\frac{|(\zeta;\zeta^2)_k(\zeta^6;\zeta^6)_k(a\zeta^3;\zeta^6)_k(\zeta^3/a;\zeta^6)_k|}
{|(1-\zeta)\,(\zeta^2;\zeta^2)_k(\zeta^3;\zeta^6)_k(a\zeta^6;\zeta^6)_k(\zeta^6/a;\zeta^6)_k|}+1
$$
as $q\to\zeta$. As in the proof of Theorem~\ref{th:baby2} we conclude that
\begin{equation}
\sum_{k=0}^{d-1}c_\zeta(k)=0,
\label{sum1}
\end{equation}
while Lemma~\ref{lem-other} applied for $d$ in place of $n$ after the substitution $q=\zeta$ results in
\begin{equation}
\sum_{k=0}^{(d-1)/2}c_\zeta(k)=0.
\label{sum1/2}
\end{equation}
These equalities imply
$$
\sum_{k=0}^{n-1}c_\zeta(k)=\sum_{k=0}^{(n-1)/2}c_\zeta(k)=0
$$
for any $d$-th primitive root of unity $\zeta$ with $d\mid n$ and $d>1$, so that
\begin{equation}
\sum_{k=0}^{m}\frac{(aq;q^2)_k (q/a;q^2)_k (q;q^2)_{2k}}{(q^2;q^2)_{2k}(aq^6;q^6)_k (q^6/a;q^6)_k }[8k+1]q^{2k^2}
\equiv0\equiv q^{-(n-1)/2}[n]\biggl(\frac{-3}{n}\biggr) \pmod{[n]}
\label{any-a}
\end{equation}
for both $m=n-1$ and $m=(n-1)/2$.
Furthermore, it follows from Lemma~\ref{lem-new} that
$$
\sum_{k=0}^m\frac{(aq;q^2)_k (q/a;q^2)_k (q;q^2)_{2k}}{(q^2;q^2)_{2k}(aq^6;q^6)_k (q^6/a;q^6)_k }[8k+1]q^{2k^2}
=q^{-(n-1)/2}[n]\biggl(\frac{-3}{n}\biggr)
$$
when $a=q^n$ or $a=q^{-n}$, again for both $m=n-1$ and $m=(n-1)/2$.
This implies that the congruences
$$
\sum_{k=0}^m\frac{(aq;q^2)_k (q/a;q^2)_k (q;q^2)_{2k}}{(q^2;q^2)_{2k}(aq^6;q^6)_k (q^6/a;q^6)_k }[8k+1]q^{2k^2}
\equiv q^{-(n-1)/2}[n]\biggl(\frac{-3}{n}\biggr)
$$
hold modulo $1-aq^n$ and $a-q^n$.
Since $[n]$, $1-aq^n$ and $a-q^n$ are relatively prime polynomials, we obtain \eqref{q4a-new} and \eqref{q4b-new}.
\end{proof}

Notice that our proofs of equations \eqref{sum1} and \eqref{sum1/2} are based on two different arguments\,---\,somewhat
less uniform compared to the earlier derivation of similar equalities in the proofs of Theorems~\ref{th:baby1} and~\ref{th:baby2}. We also stress on the fact that the congruences \eqref{any-a} obtained are valid for \emph{any} $a$,
including $a=1$:
\begin{equation}
\sum_{k=0}^{m}\frac{(q;q^2)_k^2 (q;q^2)_{2k}}{(q^2;q^2)_{2k} (q^6;q^6)_k^2 }[8k+1]q^{2k^2}
\equiv0\equiv q^{-(n-1)/2}[n]\biggl(\frac{-3}{n}\biggr) \pmod{[n]}
\label{one-a}
\end{equation}
hold for $m=n-1$ and $m=(n-1)/2$.

\begin{proof}[Proof of Theorem \textup{\ref{thm:8k+1}}]
The denominators of \eqref{q4a-new} and \eqref{q4b-new} related to $a$
are the factors $(aq^6;q^6)_{n-1}(q^6/a;q^6)_{n-1}$ and $(aq^6;q^6)_{(n-1)/2}(q^6/a;q^6)_{(n-1)/2}$, respectively;
their limits as $a\to1$ are relatively prime to $\Phi_n(q)$, since $n$ is coprime with~$6$.
On the other hand, the limit of $(1-aq^n)(a-q^n)$ as $a\to1$ has the factor $\Phi_n(q)^2$.
Thus, letting $a\to1$ in \eqref{q4a-new} and \eqref{q4b-new}, we see that \eqref{q4a} and \eqref{q4b} are true modulo $\Phi_n(q)^3$. At the same time, from \eqref{one-a} they are also valid modulo $[n]$.
This completes the proof of~\eqref{q4a} and~\eqref{q4b}.
\end{proof}

\section{More $q$-supercongruences}
\label{s4}

Throughout this section, $m$ always stands for $n-1$ or $(n-1)/2$. We shall give  generalizations of
some known $q$-supercongruences and also confirm some conjectures on $q$-analogues of Ramanujan-type supercongruences in \cite{Guo1, Guo2, Guo3, Guo4, Guo2018}.

\subsection{Two congruences of Van Hamme}
\label{s4.1}
We start with the following $q$-super\-cong\-ruence from \cite[Theorem 1.2]{Guo2018}:
\begin{align}
\sum_{k=0}^{(n-1)/2}(-1)^k q^{k^2}[4k+1]\frac{(q;q^2)_k^3}{(q^2;q^2)_k^3}
\equiv q^{(n-1)^2/4}[n] (-1)^{(n-1)/2}\pmod{[n]\Phi_n(q)^2}
\label{eq:q-Guo2018-1}
\end{align}
for odd $n$, which is a $q$-analogue of the (B.2) supercongruence of Van Hamme \cite{Hamme}.
Along the same lines as Theorem \ref{thm:8k+1-a}, we have the following generalization of \eqref{eq:q-Guo2018-1}.

\begin{theorem}
\label{th:4.1}
Let $n$ be a positive odd integer. Then, modulo $[n](1-aq^n)(a-q^n)$,
\begin{equation}
\sum_{k=0}^{m}(-1)^k q^{k^2}[4k+1]\frac{(aq;q^2)_k(q/a;q^2)_k(q;q^2)_k}{(aq^2;q^2)_k(q^2/a;q^2)_k(q^2;q^2)_k}
\equiv q^{(n-1)^2/4}[n] (-1)^{(n-1)/2}.
\label{eq:q-Guo2018-2}
\end{equation}
\end{theorem}

As we have seen in the proof of Theorem \ref{thm:8k+1}, the modulus $[n](1-aq^n)(a-q^n)$ transforms to contain the factor $\Phi_n(q)^3$ as $a\to1$ and from the (sketch of) proof below the congruence \eqref{eq:q-Guo2018-2} is true for $a=1$ modulo~$[n]$.
Therefore, the congruence \eqref{eq:q-Guo2018-2} reduces to \eqref{eq:q-Guo2018-1} when $m=(n-1)/2$
but it also confirms \cite[Conjecture 5.1]{Guo2018} when $m=n-1$, as $a\to1$.

\begin{proof}[Sketch of proof]
The terminating case of the sum of the very-well-poised ${}_6\phi_5$ series,
\begin{multline}
\sum_{k=0}^\infty\frac{(1-aq^{2k})(a;q)_k(b;q)_k(c;q)_k(d;q)_k}
{(1-a)(q;q)_k(aq/b;q)_k(aq/c;q)_k(aq/d;q)_k}\biggl(\frac{aq}{bcd}\biggr)^k
\\
=\frac{(aq;q)_\infty(aq/bc;q)_\infty(aq/bd;q)_\infty(aq/cd;q)_\infty}
{(aq/b;q)_\infty(aq/c;q)_\infty(aq/d;q)_\infty(aq/bcd;q)_\infty},
\label{Eq:6phi5}
\end{multline}
reads
\begin{equation}
\sum_{k=0}^N\frac{(1-aq^{2k})(a;q)_k(b;q)_k(c;q)_k(q^{-N};q)_k}
{(1-a)(q;q)_k(aq/b;q)_k(aq/c;q)_k(aq^{N+1};q)_k}\biggl(\frac{aq^{N+1}}{bc}\biggr)^k
=\frac{(aq;q)_N (aq/bc;q)_N}{(aq/b;q)_N(aq/c;q)_N}
\label{eq:6phi5}
\end{equation}
(see \cite[Appendix, eqs.~(II.20) and (II.21)]{GR04}).
Letting $N\to\infty$, $q\to q^2$, $a=q$ in~\eqref{eq:6phi5}, then $b=aq$ and $c=q/a$ we obtain
$$
\sum_{k=0}^{\infty}(-1)^k q^{k^2}[4k+1]\frac{(aq;q^2)_k(q/a;q^2)_k(q;q^2)_k}{(aq^2;q^2)_k(q^2/a;q^2)_k(q^2;q^2)_k}
=\frac{(q;q^2)_\infty (q^3;q^2)_\infty}{(aq^2;q^2)_\infty (q^2/a;q^2)_\infty}.
$$
Considering the limits of both sides as $q\to\zeta$, a $d$-th root of unity for each $d\mid n$, we see that \eqref{eq:q-Guo2018-2} holds modulo $[n]$, while letting $a=q^n$ we conclude that \eqref{eq:q-Guo2018-2} holds modulo $(1-aq^n)(a-q^n)$.
\end{proof}

In \cite{GuoWang}, the first author and Wang obtained a $q$-analogue of \cite[Theorem 1.1 with $r=1$]{Long}: for odd $n$,
\begin{equation*}
\sum_{k=0}^{(n-1)/2}[4k+1]\frac{(q;q^2)_k^4}{(q^2;q^2)_k^4}
\equiv q^{(1-n)/2}[n]+\frac{(n^2-1)(1-q)^2}{24}\,q^{(1-n)/2}[n]^3 \pmod{[n]\Phi_n(q)^3},
\end{equation*}
which modulo $[n]\Phi_n(q)^2$ corresponds to the following $q$-analogue of the (C.2) supercongruence of Van Hamme \cite{Hamme}:
\begin{equation}
\sum_{k=0}^{(n-1)/2}[4k+1]\frac{(q;q^2)_k^4}{(q^2;q^2)_k^4}
\equiv q^{(1-n)/2}[n] \pmod{[n]\Phi_n(q)^2}.
\label{eq:q-Long-r}
\end{equation}

We have the following two-parameter common generalization of
\eqref{eq:q-Guo2018-2} (corresponding to $c\to 0$) and \eqref{eq:q-Long-r} (corresponding to $c=1$).
\begin{theorem}
\label{th:4.2}
Let $n$ be a positive odd integer. Then, modulo $[n](1-aq^n)(a-q^n)$,
\begin{equation}
\sum_{k=0}^{m}[4k+1]\frac{(aq;q^2)_k (q/a;q^2)_k (q/c;q^2)_k (q;q^2)_k}
{(aq^2;q^2)_k(q^2/a;q^2)_k (cq^2;q^2)_k (q^2;q^2)_k}\,c^k
\equiv\frac{(c/q)^{(n-1)/2} (q^2/c;q^2)_{(n-1)/2}}{(cq^2;q^2)_{(n-1)/2}}\,[n].
\label{eq:q-Long-gen}
\end{equation}
\end{theorem}

\begin{proof}[Sketch of proof]
Taking $q\to q^2$, $a=q$ in~\eqref{Eq:6phi5}, then $c=aq$ and $d=q/a$ we obtain
\begin{multline*}
\sum_{k=0}^\infty[4k+1]\frac{(aq;q^2)_k (q/a;q^2)_k (b;q^2)_k (q;q^2)_k}
{(aq^2;q^2)_k (q^2/a;q^2)_k (q^3/b;q^2)_k (q^2;q^2)_k}\biggl(\frac{q}{b}\biggr)^k
\\
=\frac{(q^3;q^2)_\infty(q;q^2)_\infty(q^2/ab;q^2)_\infty(aq^2/b;q^2)_\infty}
{(aq^2;q^2)_\infty(q^2/a;q^2)_\infty(q^3/b;q^2)_\infty(q/b;q^2)_\infty}.
\end{multline*}
For a root of unity $\zeta$ of odd degree $d\mid n$, the limit of the right-hand side is 0 as $q\to\zeta$,
because of the presence of the factor $(q;q^2)_\infty$. Letting $q$ tend to $\zeta$ on the left-hand side results in
\begin{align*}
&
\sum_{\ell=0}^\infty \frac{(b;\zeta^2)_{\ell d}}{b^{\ell d}(\zeta^3/b;\zeta^2)_{\ell d}}\binom{2\ell}{\ell}\frac1{4^\ell}
\sum_{k=0}^{d-1}[4k+1]_\zeta\frac{(a\zeta;\zeta^2)_k (\zeta/a;\zeta^2)_k (b;\zeta^2)_k (\zeta;\zeta^2)_k}
{(a\zeta^2;\zeta^2)_k (\zeta^2/a;\zeta^2)_k (\zeta^3/b;\zeta^2)_k (\zeta^2;\zeta^2)_k}\biggl(\frac{\zeta}{b}\biggr)^k
\\ &\quad
=\sum_{\ell=0}^\infty \binom{2\ell}{\ell}\frac1{4^\ell}
\sum_{k=0}^{(d-1)/2}[4k+1]_\zeta\frac{(a\zeta;\zeta^2)_k (\zeta/a;\zeta^2)_k (b;\zeta^2)_k (\zeta;\zeta^2)_k}
{(a\zeta^2;\zeta^2)_k (\zeta^2/a;\zeta^2)_k (\zeta^3/b;\zeta^2)_k (\zeta^2;\zeta^2)_k}\biggl(\frac{\zeta}{b}\biggr)^k
\end{align*}
implying
$$
\sum_{k=0}^{m}[4k+1]\frac{(aq;q^2)_k (q/a;q^2)_k (b;q^2)_k (q;q^2)_k}
{(aq^2;q^2)_k (q^2/a;q^2)_k (q^3/b;q^2)_k (q^2;q^2)_k}\biggl(\frac{q}{b}\biggr)^k
\equiv0\pmod{[n]}
$$
for any $b\ne0$, in particular, for $b=q/c$.

Finally, the congruence \eqref{eq:q-Long-gen} modulo $1-aq^n$ and $a-q^n$ follows from the summation
\begin{equation*}
\sum_{k=0}^m[4k+1]\frac{(q^{1-n};q^2)_k  (q^{1+n};q^2)_k (q;q^2)_k (q/c;q^2)_k}
{(q^{2-n};q^2)_k (q^{2+n};q^2)_k (q^2;q^2)_k (cq^2;q^2)_k}\,c^k
=\frac{(c/q)^{(n-1)/2} (q^2/c;q^2)_{(n-1)/2}}{(cq^2;q^2)_{(n-1)/2}}\,[n],
\end{equation*}
which is the specialization $q\to q^2$, $a=q$, $c=q^{1+n}$, $N=(n-1)/2$ and $b=q/c$ of~\eqref{eq:6phi5}.
\end{proof}

\subsection{Another two congruences from Van Hamme's list}
\label{s4.2}
The following $q$-super\-congruence conjectured in \cite[eqs. (1.4) and (1.5)]{Guo1} is a partial $q$-analogue of the (J.2) supercongruence of Van Hamme \cite{Hamme}:
\begin{equation}
\sum_{k=0}^{m}q^{k^2}[6k+1]\frac{(q;q^2)_k^2 (q^2;q^4)_k }{(q^4;q^4)_k^3}
\equiv (-q)^{(1-n)/2}[n] \pmod{[n]\Phi_n(q)^2}.
\label{eq:guo1-1}
\end{equation}
It is established modulo $[n]\Phi_n(q)$ in \cite[Theorem 1.3]{Guo1}.
Here we confirm \eqref{eq:guo1-1} by showing the following more general form, which is also a generalization of Theorem~\ref{th:baby2}.

\begin{theorem}
\label{thm:q-Hamme-J2}
Let $n$ be a positive odd integer. Then, modulo $[n](1-aq^n)(a-q^n)$,
\begin{equation*}
\sum_{k=0}^{m}q^{k^2}[6k+1]\frac{(aq;q^2)_k (q/a;q^2)_k (q^2;q^4)_k }{(aq^4;q^4)_k(q^4/a;q^4)_k(q^4;q^4)_k} \equiv (-q)^{(1-n)/2}[n].
\end{equation*}
\end{theorem}

\begin{proof}[Sketch of proof]
Our derivation in \cite{GZ18} of the related $q$-analogue of a formula for $1/\pi$ uses the formula \cite[eq.~(4.6)]{Ra93}:
\begin{align}
&
\sum_{k=0}^\infty\frac{(a;q)_k (1-aq^{3k})(d;q)_k(q/d;q)_k(b;q^2)_k}
{(q^2;q^2)_k (1-a)(aq^2/d;q^2)_k (adq;q^2)_k (aq/b;q)_k}\frac{a^k q^{k+1\choose 2}}{b^k}
\notag\\ &\qquad
=\frac{(aq;q^2)_\infty (aq^2;q^2)_\infty (adq/b;q^2)_\infty (aq^2/bd;q^2)_\infty}
{(aq/b;q^2)_\infty (aq^2/b;q^2)_\infty (aq^2/d;q^2)_\infty (adq;q^2)_\infty}.
\label{quadratic}
\end{align}
Letting $q\to q^2$ and taking $a=q$, $d=aq$ and $b=q^2$, we are led to
$$
\sum_{k=0}^{\infty}q^{k^2}[6k+1]\frac{(aq;q^2)_k (q/a;q^2)_k (q^2;q^4)_k }{(aq^4;q^4)_k(q^4/a;q^4)_k(q^4;q^4)_k}
=\frac{(aq^2;q^4)_{\infty}(q^2/a;q^4)_{\infty}}{(1-q)(aq^4;q^4)_{\infty} (q^4/a;q^4)_{\infty}}.
$$
The rest is similar to the proof of Theorem \ref{thm:8k+1-a}.
\end{proof}

We complement the result by the following complete $q$-analogue of Van Hamme's supercongruence (J.2)
(see \cite[Conjecture 1.1]{Guo1}), which remains open:
\begin{multline*}
\sum_{k=0}^{(n-1)/2}q^{k^2}[6k+1]\frac{(q;q^2)_k^2 (q^2;q^4)_k }{(q^4;q^4)_k^3}
\\
\equiv (-q)^{(1-n)/2}[n]+\frac{(n^2-1)(1-q)^2}{24}\,(-q)^{(1-n)/2}[n]^3 \pmod{[n]\Phi_n(q)^3}
\end{multline*}
for odd $n$.

Similarly, we have the following $q$-analogue of the (L.2) supercongruence of Van Hamme \cite{Hamme}:
\begin{equation}
\sum_{k=0}^{m}(-1)^k [6k+1]\frac{(q;q^2)_k^3}{(q^4;q^4)_k^3}
\equiv (-q)^{-(n-1)(n+5)/8}[n] \pmod{[n]\Phi_n(q)^2},
\label{eq:Guo2-1}
\end{equation}
which is conjectured in \cite[Conjecture 1.1]{Guo2} and proved in \cite[Theorem 1.2]{Guo2} for special cases.
Here we are able to confirm \eqref{eq:Guo2-1} in the full generality as a consequence of the following result, which is also a generalization of Theorem~\ref{th:baby1}.

\begin{theorem}
\label{thm:q-Hamme-L2}
Let $n$ be a positive odd integer. Then, modulo $[n](1-aq^n)(a-q^n)$,
$$
\sum_{k=0}^{m}(-1)^k [6k+1]\frac{(aq;q^2)_k(q/a;q^2)_k(q;q^2)_k }{(aq^4;q^4)_k(q^4/a;q^4)_k(q^4;q^4)_k}
\equiv (-q)^{-(n-1)(n+5)/8}[n].
$$
\end{theorem}

\begin{proof}[Sketch of proof]
Replacing $q$ by $q^{-1}$, we see that the desired congruence is equivalent to
$$
\sum_{k=0}^{m}(-1)^k [6k+1]\frac{(aq;q^2)_k(q/a;q^2)_k(q;q^2)_k q^{3k^2}}
{(aq^4;q^4)_k(q^4/a;q^4)_k(q^4;q^4)_k}
\equiv (-q)^{(n-1)(n-3)/8}[n].
$$
Letting $q\to q^2$ and $b\to\infty$ in \eqref{quadratic}, then taking $a=q$ and $d=aq$ we obtain
$$
\sum_{k=0}^{\infty}(-1)^k [6k+1]\frac{(aq;q^2)_k(q/a;q^2)_k(q;q^2)_k q^{3k^2} }
{(aq^4;q^4)_k(q^4/a;q^4)_k(q^4;q^4)_k}
=\frac{(q^3;q^4)_{\infty}(q^5;q^4)_{\infty}}{(aq^4;q^4)_{\infty} (q^4/a;q^4)_{\infty}}.
$$
The remaining argument is as before.
\end{proof}

We also have a common generalization of Theorems \ref{thm:q-Hamme-J2} and \ref{thm:q-Hamme-L2} as follows.

\begin{theorem}
\label{thm:q-Hamme-J2L2}
Let $n\equiv r\pmod4$ be a positive odd integer, where $r=\pm1$. Then, modulo $[n](1-aq^n)(a-q^n)$,
\begin{align*}
&
\sum_{k=0}^{m}[6k+1]\frac{(aq;q^2)_k (q/a;q^2)_k (q;q^2)_k (b;q^4)_k q^{k^2+2k}}
{(aq^4;q^4)_k(q^4/a;q^4)_k(q^4;q^4)_k (q^3/b;q^2)_k b^k}
\\ &\qquad
\equiv \frac{(q^rb;q^4)_{(n-r)/4}}{(q^{4+r}/b;q^4)_{(n-r)/4}}\, b^{-(n-r)/4}(-q)^{(1-r)/2} [n].
\end{align*}
\end{theorem}

\begin{proof}[Sketch of proof]
This follows along the lines of our previous proofs, for the specialization $q\to q^2$, $a=q$ and $d=aq$ of the quadratic summation~\eqref{quadratic}:
\begin{align*}
&
\sum_{k=0}^{\infty}[6k+1]\frac{(aq;q^2)_k (q/a;q^2)_k (q;q^2)_k (b;q^4)_k q^{k^2+2k}}
{(aq^4;q^4)_k(q^4/a;q^4)_k(q^4;q^4)_k (q^3/b;q^2)_k b^k}
\\ &\qquad
=\frac{(q^3;q^4)_\infty (q^5;q^4)_\infty (aq^4/b;q^2)_\infty (q^4/ab;q^2)_\infty}
{(q^3/b;q^2)_\infty (q^5/b;q^2)_\infty (aq^4;q^2)_\infty (q^4/a;q^2)_\infty}.
\qedhere
\end{align*}
\end{proof}

\subsection{`Divergent' congruences}
\label{s4.3}
The first author obtained in \cite{Guo4} the following $q$-analogues of two `divergent' Rama\-nujan-type supercongruences of  Guillera and the second author \cite{GuZu}:
\begin{align}
\sum_{k=0}^{m}[3k+1]\frac{(q;q^2)_k^3 q^{-{k+1\choose 2} } }{(q;q)_k^2 (q^2;q^2)_k}
&\equiv q^{(1-n)/2}[n] \pmod{[n]\Phi_n(q)^2},
\label{eq:q-div-WZ-1}
\\
\sum_{k=0}^{n-1}(-1)^k [3k+1]\frac{(q;q^2)_k^3}{(q;q)_k^3}
&\equiv q^{(n-1)^2/4}[n](-1)^{(n-1)/2} \pmod{[n]\Phi_n(q)^2}.
\label{eq:q-Zudilin-3}
\end{align}
For both cases, the corresponding infinite hypergeometric sums diverge. Observing their connection
to Rahman's quadratic transformation \cite[eq.~(3.12)]{Ra93} (also recorded in \cite[eq.~(3.8.13)]{GR04})
we have arrived numerically at the following three-parameter common generalization of \eqref{eq:q-div-WZ-1}
and~\eqref{eq:q-Zudilin-3}.

\begin{conjecture}
\label{conj:4.4}
Let $n$ be a positive odd integer. Then, modulo $[n](1-aq^n)(a-q^n)$,
\begin{align}
&
\sum_{k=0}^m[3k+1]\,\frac{(aq;q^2)_k(q/a;q^2)_k(q;q^2)_k(q/b;q)_k(q/c;q)_k(bc;q)_k\,q^k}
{(aq;q)_k(q/a;q)_k(q;q)_k(bq^2;q^2)_k(cq^2;q^2)_k(q^3/bc;q^2)_k}
\nonumber\\ &\qquad
\equiv\frac{(bcq;q^2)_{(n-1)/2} (q^2/b;q^2)_{(n-1)/2} (q^2/c;q^2)_{(n-1)/2}}
{(q^3/bc;q^2)_{(n-1)/2}(bq^2;q^2)_{(n-1)/2}(cq^2;q^2)_{(n-1)/2}}\,[n].
\label{eq:conj-bc}
\end{align}
\end{conjecture}

Note that the infinite sum for the left-hand side of \eqref{eq:conj-bc} is the specialization $a=q$
of the left-hand side in \cite[eq.~(3.8.13)]{GR04},
where one further sets $d=aq$ and replaces $b,c$ with $q/b,q/c$, respectively:
\begin{align}
&
\sum_{k=0}^{\infty}[3k+1]\,\frac{(aq;q^2)_k(q/a;q^2)_k(q;q^2)_k(q/b;q)_k(q/c;q)_k(bc;q)_k\,q^k}
{(aq;q)_k(q/a;q)_k(q;q)_k(bq^2;q^2)_k(cq^2;q^2)_k(q^3/bc;q^2)_k}
\nonumber\\ &\qquad
=\frac{(q^2/b;q^2)_\infty(q^2/c;q^2)_\infty(bcq;q^2)_\infty}
{(1-q)\,(bq^2;q^2)_\infty(cq^2;q^2)_\infty(q^3/bc;q^2)_\infty}
\sum_{k=0}^\infty\frac{(q/b;q^2)_k(q/c;q^2)_k(bc;q^2)_k\,q^{2k}}{(q^2;q^2)_k(aq^2;q^2)_k(q^2/a;q^2)_k}.
\label{eq:Rah93}
\end{align}

The $a$-parametric versions of the congruences \eqref{eq:q-div-WZ-1} and~\eqref{eq:q-Zudilin-3}
are obtained from \eqref{eq:conj-bc} by setting $b\to0$ followed by $c=1$ and $c\to0$, respectively. We cannot establish
this numerical observation in its entirety but we can settle two particular cases.

\begin{theorem}
\label{thm:conj-a=qn}
For $n$ a positive odd integer, the congruence \eqref{eq:conj-bc} is valid modulo $(1-aq^n)(a-q^n)$.
\end{theorem}

\begin{proof}[Sketch of proof]
For convenience, we will use here the standard notation
$$
{}_{r+3}W_{r+2}(a_0;a_1,a_2,\dots,a_r;q,z)
=\sum_{k=0}^\infty\frac{(1-a_0 q^{2k})\,(a_0;q)_k(a_1;q)_k\dotsb(a_r;q)_k\,z^k}
{(1-a_0)\,(q;q)_k(qa_0/a_1;q)_k\dotsb(qa_0/a_r;q)_k}
$$
for very-well-poised (basic) hypergeometric series.

Take $a=q^{1+2N}$. Then the transformation \cite[eq.~(3.8.14)]{GR04} applies, in which the parameters
$a$, $b$, $c$ and $f$ are replaced with our $q/b$, $q/c$, $b$ and $q$, respectively:
\begin{align*}
&
\sum_{k=0}^N[3k+1]\,\frac{(q^{2+2N};q^2)_k(q^{-2N};q^2)_k(q;q^2)_k(q/b;q)_k(q/c;q)_k(bc;q)_k\,q^k}
{(q^{2+2N};q)_k(q^{-2N};q)_k(q;q)_k(bq^2;q^2)_k(cq^2;q^2)_k(q^3/bc;q^2)_k}
\\ &\qquad
=[2N+1]\,\frac{(bcq;q^2)_N(q^2/bc;q^2)_N}{(bcq^2;q^2)_N(q^3/bc;q^2)_N}
\,{}_{10}W_9(bc;q,b,c,bc,qbc,q^{2N+2},q^{-2N};q^2,q^2)
\\ &\qquad
=[2N+1]\,\frac{(bcq;q^2)_N(q^2/bc;q^2)_N}{(bcq^2;q^2)_N(q^3/bc;q^2)_N}
\,{}_8W_7(bc;b,c,bc,q^{2N+2},q^{-2N};q^2,q^2)
\\ \intertext{(this is summable by Jackson's $q$-analogue of Dougall's $_7F_6$ sum \cite[eq.~(II.22)]{GR04})}
&\qquad
=[2N+1]\,\frac{(bcq;q^2)_N(q^2/bc;q^2)_N}{(bcq^2;q^2)_N(q^3/bc;q^2)_N}
\,\frac{(bcq^2;q^2)_N(q^2/b;q^2)_N(q^2/c;q^2)_N}{(bq^2;q^2)_N(cq^2;q^2)_N(q^2/bc;q^2)_N}
\\ &\qquad
=[2N+1]\,\frac{(bcq;q^2)_N(q^2/b;q^2)_N(q^2/c;q^2)_N}{(q^3/bc;q^2)_N(bq^2;q^2)_N(cq^2;q^2)_N}.
\end{align*}
This establishes \eqref{eq:conj-bc} simultaneously modulo $a-q^n$ and $1-aq^n$ for $n=2N+1$.
\end{proof}

\begin{theorem}
\label{thm:conj-c=1}
Let $n$ be a positive odd integer. Then, modulo $[n](1-aq^n)(a-q^n)$,
\begin{align*}
&
\sum_{k=0}^m[3k+1]\,\frac{(aq;q^2)_k(q/a;q^2)_k(b;q)_k(q/b;q)_k(q;q^2)_k\,q^k}
{(aq;q)_k(q/a;q)_k(bq^2;q^2)_k(q^3/b;q^2)_k(q^2;q^2)_k}
\\ &\qquad
\equiv \frac{(bq;q^2)_{(n-1)/2} (q^2/b;q^2)_{(n-1)/2}}{(bq^2;q^2)_{(n-1)/2}(q^3/b;q^2)_{(n-1)/2}} [n].
\end{align*}
\end{theorem}

This confirms Conjecture~\ref{conj:4.4} when $c=1$.

\begin{proof}[Sketch of proof]
In view of Theorem~\ref{thm:conj-a=qn}, we only need to verify the required congruence modulo~$[n]$.
Take $c=q^{1+2N}$ in \eqref{eq:Rah93} for $N$ a positive integer,
so that the $q$-Saalsch\"utz theorem \cite[eq.~(II.12)]{GR04} applies to the right-hand side:
\begin{align*}
&
\sum_{k=0}^{2N}[3k+1]\,\frac{(aq;q^2)_k(q/a;q^2)_k(q;q^2)_k(q/b;q)_k(q^{-2N};q)_k(bq^{1+2N};q)_k\,q^k}
{(aq;q)_k(q/a;q)_k(q;q)_k(bq^2;q^2)_k(q^{3+2N};q^2)_k(q^{2-2N}/b;q^2)_k}
\\ &\qquad
=\frac{(q^{1-2N};q^2)_\infty(q^2/b;q^2)_\infty(bq^{2+2N};q^2)_\infty}
{(1-q)\,(q^{3+2N};q^2)_\infty(q^{2+2N}/b;q^2)_\infty(bq^2;q^2)_\infty}
\,\frac{(abq;q^2)_N(aq^{1-2N}/b;q^2)_N}{(aq^2;q^2)_N(aq^{-2N};q^2)_N}
\\ &\qquad
=\frac{(-1)^N(q/b)^Nq^{-N^2}(q;q^2)_N^2(q^2/b;q^2)_N(abq;q^2)_N(bq/a;q^2)_N}
{(1-q)\,(bq^2;q^2)_N(aq^2;q^2)_N(q^2/a;q^2)_N}.
\end{align*}
Now for $d>1$ odd take a primitive $d$-th root of unity $\zeta$,
then $M>0$ odd and specialize $N$ above to be $(dM-1)/2$.
The limit of the right-hand side as $q\to\zeta$ is equal to $0$, because of the factor $(q;q^2)_N$. The limit of the left-hand side is
$$
\sum_{\ell=0}^{M-1}\frac1{2^\ell}\binom{2\ell}{\ell}
\sum_{k=0}^{d-1}[3k+1]\,\frac{(aq;q^2)_k(q/a;q^2)_k(b;q)_k(q/b;q)_k(q;q^2)_k\,q^k}
{(aq;q)_k(q/a;q)_k(bq^2;q^2)_k(q^3/b;q^2)_k(q^2;q^2)_k}\bigg|_{q=\zeta},
$$
where we use that $\zeta^{-2N}=\zeta^{1-dM}=\zeta$,
so that we conclude with the congruence
$$
\sum_{k=0}^m[3k+1]\,\frac{(aq;q^2)_k(q/a;q^2)_k(b;q)_k(q/b;q)_k(q;q^2)_k\,q^k}
{(aq;q)_k(q/a;q)_k(bq^2;q^2)_k(q^3/b;q^2)_k(q^2;q^2)_k} \equiv0\pmod{[n]}
$$
for odd $n$.
\end{proof}

Although Theorem~\ref{thm:conj-c=1} implies the $a$-parametric version of \eqref{eq:q-div-WZ-1},
there is a stronger version of the latter congruence (see \cite[Conjecture 7.1]{Guo4}) which remains open:
if $n$~is odd then
\begin{align*}
&
\sum_{k=0}^{n-1}[3k+1]\frac{(q;q^2)_k^3 q^{-{k+1\choose 2} } }{(q;q)_k^2 (q^2;q^2)_k}
\\ &\qquad
\equiv q^{(1-n)/2}[n]+\frac{(n^2-1)(1-q)^2}{24}\,q^{(1-n)/2}[n]^3 \pmod{[n]\Phi_n(q)^3}.
\end{align*}

Some other specializations of Theorem~\ref{thm:conj-c=1} are interesting by themselves.
For example, the choice $q\to q^2$, $b=q$ and $a\to1$ leads us to
$$
\sum_{k=0}^{(n-1)/2}[3k+1]_{q^2}\,\frac{(q;q^2)_k^2 (q^2;q^4)_k^3 \,q^{2k}} {(q^2;q^2)_k^2(q^4;q^4)_k(q^5;q^4)_k^2}
\equiv 0\pmod{\Phi_n(q)^3}
$$
for a positive integer $n\equiv 3\pmod{4}$. Notice the equivalence of the congruences for $m=(n-1)/2$ and $m=n-1$ in this special case. This, in turn, implies that for a prime $p$ congruent to 3 modulo 4 we have
$$
\sum_{k=0}^{(p-1)/2}(3k+1)\,\frac{(\frac{1}{2})_k^5}{(1)_k^3 (\frac{5}{4})_k^2}\equiv 0\pmod{p^3}.
$$

\subsection{Generalized Van Hamme's congruences}
\label{s4.4}
In \cite[Theorem 1.3]{Guo3}, a uniform version of $q$-analogues of the (B.2), (E.2) and (F.2) supercongruences of Van Hamme are given.
The following result provides a generalization of the $q$-supercongruence that depends on an additional parameter $a$.

\begin{theorem}
\label{conj:one}
Let $d$ be a positive integer and let $r$ be an integer with $\gcd(r,d)=1$. Then, for any positive integer $n\equiv r\pmod{d}$
such that $n+d-nd\le r\le n$, we have
\begin{multline}
\sum_{k=0}^{M}(-1)^k q^{d{k+1\choose 2}-rk}[2dk+r]\frac{(aq^r;q^d)_k(q^r/a;q^d)_k(q^r;q^d)_k }{(aq^{d};q^{d})_k (q^{d}/a;q^{d})_k (q^{d};q^{d})_k }
\\
\equiv q^{(n-r)(n-d+r)/(2d)}[n] (-1)^{(n-r)/d}\pmod{[n](1-aq^n)(a-q^n)},
\label{eq:last-one}
\end{multline}
where $M=(n-r)/d$ or $M=n-1$.
\end{theorem}

Note that the $a\to 1$ and $M=n-1$ case of \eqref{eq:last-one} partially confirms \cite[Conjecture 5.1]{Guo3}.
In particular, if $p$ is a prime with $p^s\equiv 1\pmod{d}$, then
$$
\sum_{k=0}^{p^s-1}(-1)^k (2dk+1)\frac{(\frac{1}{d})_k^3}{k!^3}
\equiv p^s  (-1)^{(p^s-1)/d}\pmod{p^{s+2}}.
$$

\begin{proof}[Sketch of proof]
Letting $N\to\infty$, $q\to q^d$, $a=q^r$ in \eqref{eq:6phi5}, followed by $b=aq^r$ and $c=q^r/a$, we obtain
\begin{equation}
\sum_{k=0}^\infty(-1)^kq^{d\binom{k+1}{2}-rk}[2dk+r]\frac{(aq^r;q^d)_k(q^r/a;q^d)_k(q^r;q^d)_k}
{(aq^d;q^d)_k(q^d/a;q^d)_k(q^d;q^d)_k}
=[r]\,\frac{(q^{d-r};q^d)_\infty(q^{d+r};q^d)_\infty}{(aq^d;q^d)_\infty(q^d/a;q^d)_\infty}.
\label{eq:aqrd}
\end{equation}
Let $\zeta$ be an $e$-th primitive root of unity with $e\mid n$. By the hypothesis of the theorem, we see that $\gcd(n,d)=1$ and so $\gcd(e,d)=1$. This means that
there is one and only one number divisible by $e$ in the arithmetic progression $r, r+d,\ldots,r+(e-1)d$. Denote this number by $r+ud=ve$. Then by L'H\^opital's rule we see that
\begin{align*}
\lim_{q\to\zeta}\frac{(q^r;q^d)_{\ell e+k}}{(q^d;q^d)_{\ell e+k}}
&=\frac{v\,(v+d)\dotsb(v+(\ell-1)d)}{d\cdot2d\,\dotsb\,\ell d}\,\lim_{q\to\zeta}\frac{(q^r;q^d)_k}{(q^d;q^d)_k}
\\
&={v/d+\ell-1\choose\ell}\lim_{q\to\zeta}\frac{(q^r;q^d)_k}{(q^d;q^d)_k}
\end{align*}
for $\ell\ge 0$ and $0\le k<e$. It is clear that $v\neq d$. Since
$$
\sum_{\ell=0}^\infty (-1)^\ell {v/d+\ell-1\choose\ell}
=\begin{cases} 2^{-v/d}\neq 0 &\text{if $v<d$},\\
\infty &\text{if $v>d$},
\end{cases}
$$
the proof of \eqref{eq:last-one} modulo $[n]$ follows the lines of
the proofs of Theorems~\ref{th:baby1} and \ref{th:baby2}.

Finally, the congruence \eqref{eq:last-one} modulo $1-aq^n$ and $a-q^n$ follows from setting $a=q^{-n}$ in~\eqref{eq:aqrd}:
\begin{align*}
&
\sum_{k=0}^M(-1)^kq^{d\binom{k+1}{2}-rk}[2dk+r]\frac{(q^{r-n};q^d)_k(q^{r+n};q^d)_k(q^r;q^d)_k}
{(q^{d-n};q^d)_k(q^{d+n};q^d)_k(q^d;q^d)_k}
\\ &\qquad
=[r]\,\frac{(q^{d-r};q^d)_\infty(q^{d+r};q^d)_\infty}{(q^{d-n};q^d)_\infty(q^{d+n};q^d)_\infty}
=[r]\,\frac{(q^{d+r};q^d)_{(n-r)/d}}{(q^{d-n};q^d)_{(n-r)/d}}
\displaybreak[2]\\ &\qquad
=(-1)^{(n-r)/d}q^{(n-r)(n-d+r)/(2d)}[r]\,\frac{(q^{d+r};q^d)_{(n-r)/d}}{(q^r;q^d)_{(n-r)/d}}
\\ &\qquad
=q^{(n-r)(n-d+r)/(2d)}[n] (-1)^{(n-r)/d}.
\end{align*}
Note that the conditions $n\ge r$ and $n\equiv r\pmod{d}$ imply that the left-hand side terminates at $k=(n-r)/d$,
while the hypothesis $n+d-nd\le r$ means that $(n-r)/d\le n-1$.
\end{proof}

Using the above basic hypergeometric series identity, we can also prove the following generalization of \cite[Theorem 1.5]{Guo3}.

\begin{theorem}\label{conj:two}
Let $d$ be a positive integer and let $r$ be an integer with $\gcd(r,d)=1$.
Then, for any positive integer $n\equiv-r\pmod{d}$ such that $d-n\le r\le (d-1)n$, we have
\begin{multline}
\sum_{k=0}^{M}(-1)^k q^{d{k+1\choose 2}-rk}[2dk+r]\frac{(aq^r;q^d)_k(q^r/a;q^d)_k(q^r;q^d)_k }{(aq^{d};q^{d})_k (q^{d}/a;q^{d})_k (q^{d};q^{d})_k }
\\
\equiv q^{(nd-n-r)(nd-n-d+r)/(2d)} [(d-1)n]
(-1)^{((d-1)n-r)/d} \pmod{[n](1-aq^n)(a-q^n)},
\label{eq:last-two}
\end{multline}
where $M=((d-1)n-r)/d$ or $M=n-1$.
\end{theorem}

Note that the $a\to 1$ and $M=n-1$ case of \eqref{eq:last-two} confirms \cite[Conjecture 5.2]{Guo3}. In particular,
if $p$ is a prime satisfying $p^s\equiv -1\pmod{d}$, then
$$
\sum_{k=0}^{p^s-1}(-1)^k (2dk+1)\frac{(\frac{1}{d})_k^3}{k!^3}
\equiv (d-1)p^s  (-1)^{((d-1)p^s-1)/d}\pmod{p^{s+2}}.
$$

\subsection{A strange congruence}
\label{s4.5}
In \cite[Conjecture 7.2]{Guo4}, the following strange conjecture was proposed: for any positive integer $n$ with $n\equiv 1\pmod 4$,
\begin{align}
\sum_{k=0}^{(n-1)/2}[4k+1]\frac{(q;q^2)_k^3}{(q^2;q^2)_k^3} q^{k(n^2-2nk-n-2)/4}\equiv 0\pmod{\Phi_n(q)^2}.  \label{strange-1}
\end{align}
Note that $k(n^2-2nk-n-2)/4$ is a two-variable polynomial of degree~3. Congruences of this form are very rare.
We now give a related parametric result.

\begin{theorem}
\label{thm:strange}
Let $n\equiv1\pmod 4$ be a positive integer. Then
\begin{equation}
\sum_{k=0}^{(n-1)/2}[4k+1]\frac{(aq;q^2)_k (q/a;q^2)_k  (q;q^2)_k q^{(n-1)k/2}}
{(aq^2;q^2)_k(q^2/a;q^2)_k (q^2;q^2)_k}\equiv 0\pmod{\Phi_n(q)}.
\label{strange-2}
\end{equation}
\end{theorem}

It is easy to see that the term $q^{(n-1)k/2}$ in \eqref{strange-2} can be replaced by $q^{k(n^2-2nk-n-2)/4}$. However, we cannot replace the term $q^{k(n^2-2nk-n-2)/4}$ in \eqref{strange-1} by
$q^{(n-1)k/2}$.

\begin{proof}[Sketch of proof]
Set $q\to q^4$, $a=q^2$, $d=q^3$ in \eqref{Eq:6phi5}, then take $b=aq^2$ and $c=q^2/a$:
\begin{align*}
&
\sum_{k=0}^\infty[4k+1]_{q^2}\frac{(aq^2;q^4)_k(q^2/a;q^4)_k(q^2;q^4)_k q^{-k}}
{(aq^4;q^4)_k(q^4/a;q^4)_k(q^4;q^4)_k}
\\ &\qquad
=\frac{(q^2;q^4)_\infty(q^6;q^4)_\infty (aq;q^4)_\infty(q/a;q^4)_\infty}
{(1-q^{-1})(q^3;q^4)_\infty^2 (aq^4;q^4)_\infty(q^4/a;q^4)_\infty}.
\end{align*}
Now choose any primitive $n$-th root of unity $\zeta\ne1$ and consider the limit of both sides of the equality as $q\to\zeta$.
The right-hand side clearly tends to $0$, because of the presence of $(q^2;q^4)_\infty$; the factor $(q^3;q^4)_\infty^2$
in the denominator does not interfere, since $\zeta^{3+4j}\ne1$ when $j=0,1,2,\dots$ for the root of unity of degree $n\equiv1\pmod4$.
The standard analysis of the left-hand side leads us to
$$
\sum_{k=0}^{(n-1)/2}[4k+1]_{\zeta^2}\frac{(a\zeta^2;\zeta^4)_k(\zeta^2/a;\zeta^4)_k(\zeta^2;\zeta^4)_k \zeta^{-k}}
{(a\zeta^4;\zeta^4)_k(\zeta^4/a;\zeta^4)_k(\zeta^4;\zeta^4)_k}=0.
$$
Noticing that $\zeta^{-k}=\zeta^{(n-1)k}$ for $k=0,1,\dots,(n-1)/2$,
we have
\begin{equation}
\sum_{k=0}^{(n-1)/2}[4k+1]_{q^2}\frac{(aq^2;q^4)_k (q^2/a;q^4)_k  (q^2;q^4)_k q^{(n-1)k}}
{(aq^4;q^4)_k(q^4/a;q^4)_k (q^4;q^4)_k}\equiv 0\pmod{\Phi_n(q)}.
\label{strange-2-1}
\end{equation}
The left-hand side here remains the same if we replace $q$ with $-q$, therefore the congruence \eqref{strange-2-1} takes
place modulo $\Phi_n(-q)$ as well, hence modulo $\Phi_n(q^2)=\Phi_n(q)\Phi_n(-q)$ since $n$ is odd.
Thus, changing $q^2$ with $q$ we arrive at the congruence \eqref{strange-2}.
\end{proof}

\subsection{A congruence from the $q$-Dixon sum}
\label{s4.6}
As we have seen, truncating known basic hypergeometric series identities usually leads to new $q$-congruences, or
to `natural' candidates for $q$-analogues of those coming from non-$q$-settings. Here is another example.

\begin{theorem}
Let $n\equiv 3\pmod{4}$ be a positive integer. Then
\begin{equation}
\sum_{k=0}^{m} \frac{(1+aq^{4k+1})\,(a^2q^2;q^4)_k (bq^2;q^4)_k (cq^2;q^4)_k }
{(1+aq)\,(a^2q^4/b;q^4)_k (a^2q^4/c;q^4)_k (q^4;q^4)_k } \biggl(\frac{aq}{bc}\biggr)^k
\equiv 0\pmod{(1-a^2q^{2n})}; \label{eq:qDixon-1}
\end{equation}
in particular,
\begin{equation}
\sum_{k=0}^{(n-1)/2} \frac{(1+q^{4k+1})\, (q^2;q^4)_k^3}{(1+q)\,(q^4;q^4)_k^3} q^k
\equiv 0\pmod{\Phi_n(q)\Phi_n(-q)}.
\label{eq:qDixon-2}
\end{equation}
\end{theorem}

\begin{proof}
Taking $q\to q^4$, $a\to a^2q^2$, $b\to bq^2$ and $c\to cq^2$ in the $q$-Dixon sum \cite[eq.~(II.13)]{GR04}
we obtain
\begin{align}
&
\sum_{k=0}^\infty \frac{(1+aq^{4k+1})\,(a^2q^2;q^4)_k (bq^2;q^4)_k (cq^2;q^4)_k}
{(1+aq)\,(a^2q^4/b;q^4)_k (a^2q^4/c;q^4)_k (q^4;q^4)_k} \biggl(\frac{aq}{bc}\biggr)^k
\nonumber\\ &\qquad
=\frac{(a^2q^6;q^4)_\infty (aq^3/b;q^4)_\infty (aq^3/c;q^4)_\infty (a^2q^2/bc;q^4)_\infty}
{(a^2q^4/b;q^4)_\infty (a^2q^4/c;q^4)_\infty (aq^5;q^4)_\infty (aq/bc;q^4)_\infty}.
\label{eq:qDixon-3}
\end{align}
Since $n\equiv 3\pmod 4$, putting $a=\pm q^{-n}$ in \eqref{eq:qDixon-3} we see that the left-hand side terminates,
while the right-hand side vanishes. This proves \eqref{eq:qDixon-1}.
Letting $a,b,c\to 1$ in \eqref{eq:qDixon-1} we are led to \eqref{eq:qDixon-2}.
\end{proof}

We now provide a conjectural refinement of \eqref{eq:qDixon-2},
which is a new $q$-analogue of the (H.2) supercongruence of Van Hamme \cite{Hamme} for $p\equiv 3\pmod 4$ (corresponding to $q\to 1$).
It is also a partial $q$-analogue of the (B.2) supercongruence of Van Hamme (corresponding to $q\to -1$).

\begin{conjecture}
Let $n\equiv 3\pmod{4}$ be a positive integer. Then
$$
\sum_{k=0}^{(n-1)/2}\frac{(1+q^{4k+1})(q^2;q^4)_k^3}{(1+q)\,(q^4;q^4)_k^3} q^k
\equiv 0\pmod{\Phi_n(q)^2\Phi_n(-q)}.
$$
\end{conjecture}

\subsection{A congruence from Andrews' $q$-analogue of Gauss' $_2F_1(-1)$ sum}
\label{s4.7}
It is proved in \cite[eq.~(2.6)]{GZ2014} that, for $p$ a prime of the form $4\ell+3$,
$$
\sum_{k=0}^{p-1}\frac{(q;q^2)_k^2 q^{2k}}{(q^2;q^2)_k(q^4;q^4)_k}
=\sum_{k=0}^{p-1}\frac{(q;q^2)_k^2 q^{2k}}{(q^2;q^2)_k^2 (-q^2;q^2)_k}
\equiv 0  \pmod{[p]^2}.
$$
We now give a two-parameter extension of this congruence.

\begin{theorem}\label{eq:last-thm}
Let $n\equiv 3\pmod{4}$ be a positive integer. Then
\begin{equation}
\sum_{k=0}^{m}\frac{(aq;q^2)_k (bq;q^2)_k q^{2k}}{(q^2;q^2)_k (abq^4;q^4)_k}
\equiv 0\pmod{(1-aq^n)(1-bq^n)};
\label{eq:qab-1}
\end{equation}
in particular,
\begin{equation}
\sum_{k=0}^{(n-1)/2}\frac{(q;q^2)_k ^2 q^{2k}}{(q^2;q^2)_k (q^4;q^4)_k}\equiv 0\pmod{\Phi_n(q)^2}.  \label{eq:qab-2}
\end{equation}
\end{theorem}

\begin{proof}
Making the substitutions $q\to q^2$, $a\to aq$ and $b\to bq$ in Andrews' $q$-analogue of Gauss' $_2F_1(-1)$ sum
(see \cite{Andrews73,Andrews74} or \cite[Appendix (II.11)]{GR04}), we obtain
\begin{equation}
\sum_{k=0}^\infty \frac{(aq;q^2)_k(bq;q^2)_k
q^{k^2+k}}{(q^2;q^2)_k(abq^4;q^4)_k} =\frac{(aq^3;q^4)_\infty
(bq^3;q^4)_\infty}{(q^2;q^4)_\infty (abq^4;q^4)_\infty},
\label{eq:andrews}
\end{equation}
Since $n\equiv 3\pmod 4$, taking $a=q^{-n}$ or $b=q^{-n}$ in \eqref{eq:andrews} we see that the left-hand side terminates, while the right-hand side vanishes. This proves that
$$
\sum_{k=0}^{m}\frac{(aq;q^2)_k (bq;q^2)_k q^{k^2+k}}{(q^2;q^2)_k (abq^4;q^4)_k}
\equiv 0\pmod{(1-aq^n)(1-bq^n)},
$$
which after the rearrangement $a\to a^{-1}$, $b\to b^{-1}$ and $q\to q^{-1}$
becomes the congruence \eqref{eq:qab-1}.
Letting $a\to 1$ and $b\to 1$ in \eqref{eq:qab-1} we arrive at \eqref{eq:qab-2}.
\end{proof}

Motivated by \cite[Theorem 2.5]{GZ2014} we observe the following generalization of Theorem~\ref{eq:last-thm}.

\begin{conjecture}
Let $n$ be a positive odd integer. Then
\begin{align*}
&
\sum_{k=0}^{m}\frac{(aq;q^2)_k (bq;q^2)_k (x;q^2)_k q^{2k}}{(q^2;q^2)_k (abq^4;q^4)_k}
\\ &\qquad
\equiv (-1)^{(n-1)/2}\sum_{k=0}^{m}\frac{(aq;q^2)_k (bq;q^2)_k (-x;q^2)_k q^{2k}}
{(q^2;q^2)_k (abq^4;q^4)_k} \pmod{(1-aq^n)(1-bq^n)}.
\end{align*}
\end{conjecture}

When $x=0$ and $n\equiv 3\pmod 4$ this indeed reduces to
Theorem~\ref{eq:last-thm}. Inspired by \cite[Conjecture 7.3]{GZ2014},
we believe that the following further generalization is true as
well.

\begin{conjecture}
Let $d$, $n$ and $r$ be positive integers with $r<d$, $\gcd(d,n)=1$, and $n$ odd.
Then, modulo $(1-aq^{n\langle r/n\rangle_d})(1-bq^{n\langle(d-r)/n\rangle_d})$,
\begin{align*}
\sum_{k=0}^{n-1}\frac{(aq^r;q^d)_k (bq^{d-r};q^d)_k (x;q^{d})_k q^{dk}}{(q^d;q^d)_k (abq^{2d};q^{2d})_k}
\equiv (-1)^{\langle -r/d\rangle_n}\sum_{k=0}^{n-1}\frac{(aq^r;q^d)_k (bq^{d-r};q^d)_k (-x;q^{d})_k q^{dk}}{(q^d;q^d)_k (abq^{2d};q^{2d})_k},
\end{align*}
where $\langle z\rangle_s$ denotes the least non-negative residue of $z$ modulo $s$.
\end{conjecture}

\section{Concluding remarks and open problems}
\label{s5}

Since Ramanujan's formula \eqref{ram1} has a WZ proof \cite{Gu06},
it is natural to ask whether there is a $q$-WZ proof of its $q$-analogue~\eqref{q4}.
If this is the case, then the corresponding $q$-WZ pair will possibly lead to another proof of Theorem~\ref{thm:8k+1}.

The equality \eqref{ram1} motivates considering different families of congruences, like
\begin{align}
\sum_{k=0}^{n}(8k+1){4k\choose 2k}{2k\choose k}^2
2^{8(n-k)}3^{2(n-k)}\equiv 0 \pmod{{2n\choose n}}, \label{eq:2n-n}
\displaybreak[2]\\
\sum_{k=0}^{n}(8k+1){4k\choose 2k}{2k\choose k}^2
2^{8(n-k)}3^{2(n-k)}\equiv 0 \pmod{{3n\choose n}}, \label{eq:3n-n}
\displaybreak[2]\\
\sum_{k=0}^{n}(8k+1){4k\choose 2k}{2k\choose k}^2
2^{8(n-k)}3^{2(n-k)}\equiv 0 \pmod{{4n\choose n}}, \label{eq:4n-n}
\displaybreak[2]\\
\sum_{k=0}^{n}(8k+1){4k\choose 2k}{2k\choose k}^2
2^{8(n-k)}3^{2(n-k)}\equiv 0 \pmod{{4n\choose 2n}}, \label{eq:4n-2n}
\end{align}
which we observe numerically, and whose proofs can be accessible to the WZ method.
In view of the congruence
$$
\sum_{k=0}^{n}(-1)^k q^{k^2}[4k+1]{2k\brack k}^3 \frac{(-q;q)_n^6}{(-q;q)_k^6}
\equiv 0\pmod{(1+q^n)^2[2n+1]{2n\brack n}}.
$$
established in \cite[Theorem 1.4]{Guo2018} by the $q$-WZ method (see \cite{Guo2019,GW} for some other congruences
related to $q$-binomial coefficients), we hypothesize the truth of the following
$q$-analogues of \eqref{eq:2n-n} and \eqref{eq:3n-n}.

\begin{conjecture}
\label{conj1}
Let $n$ be a positive integer. Then
\begin{align*}
\sum_{k=0}^{n}{4k\brack 2k}{2k\brack k}^2 \frac{(-q;q)_n^4(-q;q)_{2n}^2 (q^2;q^2)_k^2 (q^6;q^6)_n^2}{(-q;q)_k^4(-q;q)_{2k}^2 (q^2;q^2)_n^2 (q^6;q^6)_k^2} [8k+1]q^{2k^2}
&\equiv 0 \pmod{{2n\brack n}}, \\
\sum_{k=0}^{n}{4k\brack 2k}{2k\brack k}^2 \frac{(-q;q)_n^4(-q;q)_{2n}^2 (q^2;q^2)_k^2 (q^6;q^6)_n^2}{(-q;q)_k^4(-q;q)_{2k}^2 (q^2;q^2)_n^2 (q^6;q^6)_k^2} [8k+1]q^{2k^2}
&\equiv 0 \pmod{{3n\brack n}}.
\end{align*}
\end{conjecture}

The expression on the left-hand sides is clearly a polynomial in $q$, and it can also be written as
\begin{align*}
\frac{(-q;q)_n^4 (-q;q)_{2n}^2 (q^6;q^6)_n^2}{ (q^2;q^2)_n^2}\sum_{k=0}^{n}\frac{(q;q^2)_k^2
(q;q^2)_{2k}}{(q^2;q^2)_{2k}(q^6;q^6)_k^2 }[8k+1]q^{2k^2}.
\end{align*}
However, similar natural $q$-analogues of \eqref{eq:4n-n} and \eqref{eq:4n-2n} do not hold in general.

As somewhat complementary to Theorem~\ref{th:4.2}, we have the following collection of parametric congruences.

\begin{conjecture}
\label{th:4.2+}
Let $d$ and $n$ be positive integers with $n\equiv -1\pmod{d}$. Then
\begin{align*}
\sum_{k=0}^{n-1}[2dk+1]\frac{(aq;q^d)_k (q/a;q^d)_k (bq;q^d)_k (q/b;q^d)_k}
{(aq^d;q^d)_k(q^d/a;q^d)_k (bq^d;q^d)_k (q^d/b;q^d)_k}q^{(d-2)k}
&\equiv 0\pmod{[n]}
\\ \intertext{and, for $d\ne2$,}
\sum_{k=0}^{n-1}[2dk+1]\frac{(aq;q^d)_k (q/a;q^d)_k (q;q^d)_k^2}
{(aq^d;q^d)_k(q^d/a;q^d)_k (q^d;q^d)_k^2}q^{(d-2)k}
&\equiv 0\pmod{[n]\Phi_n(q)}.
\end{align*}
Furthermore, for the particular case $d=2$, we also have a `shorter' congruence
\begin{equation*}
\sum_{k=0}^{(n-1)/2}[4k+1]\frac{(aq;q^2)_k (q/a;q^2)_k (bq;q^2)_k (q/b;q^2)_k}
{(aq^2;q^2)_k(q^2/a;q^2)_k (bq^2;q^2)_k (q^2/b;q^2)_k}
\equiv 0\pmod{[n]}.
\end{equation*}
\end{conjecture}

Because of
$$
\sum_{k=0}^{n-1}[2k+1]q^{-k}=[n]^2 q^{1-n},
$$
Conjecture \ref{th:4.2+} is trivially true for $d=1$.
The special case $d=2$ and $b=1$ of the conjecture is seen to be covered by Theorem~\ref{th:4.2}.

\begin{conjecture}
\label{conj-He-7}
Let $d$ and $n$ be positive integers with $d\ge 3$ and $n\equiv -1\pmod{d}$. Then
$$
\sum_{k=0}^{n-1}\frac{(a_1 q;q^d)_k (a_2q;q^d)_k \cdots (a_dq;q^d)_k q^{dk}}{(a_1q^d;q^d)_k (a_2 q^d;q^d)_k \dotsb(a_dq^d;q^d)_k}
\equiv 0 \pmod{\Phi_n(q)}
$$
and
$$
\sum_{k=0}^{n-1}\frac{(q;q^d)_k^d q^{dk}}{(q^d;q^d)_k^d}
\equiv 0 \pmod{\Phi_n(q)^2}.
$$
\end{conjecture}

The congruences in Conjecture \ref{conj-He-7} do not hold in general when $d=2$. The conjecture comes with the following companion.

\begin{conjecture}
\label{conj-He-17}
Let $d$ and $n$ be positive integers with $d\ge 2$ and $n\equiv 1\pmod{d}$. Then
\begin{equation}
\sum_{k=0}^{n-1}\frac{(a_1/q;q^d)_k (a_2/q;q^d)_k \cdots (a_d/q;q^d)_k q^{dk}}{(a_1q^d;q^d)_k (a_2 q^d;q^d)_k \dotsb(a_dq^d;q^d)_k}
\equiv 0 \pmod{\Phi_n(q)}
\label{eq:many-a/q}
\end{equation}
and
$$
\sum_{k=0}^{n-1}\frac{(q^{-1};q^d)_k^d q^{dk}}{(q^d;q^d)_k^d}
\equiv 0 \pmod{\Phi_n(q)^2}.
$$
If $d=2$, then the congruence \eqref{eq:many-a/q} further holds modulo $[n]$.
\end{conjecture}

Another related entry of Conjecture \ref{conj-He-7} for $d=4$ is as follows.

\begin{conjecture}
\label{conj-He-7-d=4}
Let $n\equiv 3\pmod{4}$ be a positive integer. Then
\begin{align*}
\sum_{k=0}^{n-1} \frac{(aq;q^4)_k (q/a;q^4)_k (q^2;q^4)_kq^{4k}}{(aq^4;q^4)_k (q^4/a;q^4)_k (q^4;q^4)_k}
&\equiv 0\pmod{\Phi_n(q)},
\\
\sum_{k=0}^{n-1} \frac{(q;q^4)_k^2 (q^2;q^4)_k q^{4k}}{(q^4;q^4)_k^3}
&\equiv 0\pmod{\Phi_n(q)^2}.
\end{align*}
\end{conjecture}

The first author and Zeng \cite[Corollary 1.2]{GZ} give a $q$-analogue of the (H.2) supercongruence of Van Hamme \cite{Hamme}. In particular, they prove that
\begin{align*}
\sum_{k=0}^{(p-1)/2}\frac{(q;q^2)_k^2 (q^2;q^4)_k q^{2k}}{(q^2;q^2)_k^2 (q^4;q^4)_k}\equiv 0\pmod{[p]^2}\quad \text{for any prime}\ p\equiv 3\pmod{4}.
\end{align*}
We now provide a related $a$-parametric version of the congruence.

\begin{conjecture}
Let $n\equiv3\pmod 4$ be a positive integer. Then
\begin{align*}
\sum_{k=0}^{(n-1)/2}\frac{(aq;q^2)_k (q/a;q^2)_k  (q^2;q^4)_k q^{2k}}{(aq^2;q^2)_k(q^2/a;q^2)_k (q^4;q^4)_k}\equiv 0\pmod{\Phi_n(q)}.
\end{align*}
\end{conjecture}

More generally, motivated by \cite[Theorem 1.3]{GZ}, we believe that the following is true.

\begin{conjecture}
Let $d$, $n$ and $r$ be positive integers with $\gcd(d,n)=1$ and $n$ odd. If the least non-negative residue of $-r/d$ modulo $n$ is odd, then
\begin{align*}
\sum_{k=0}^{(n-1)/2}\frac{(aq^r;q^d)_k (q^{d-r}/a;q^d)_k  (q^d;q^{2d})_k q^{dk}}{(aq^d;q^d)_k(q^d/a;q^d)_k (q^{2d};q^{2d})_k}\equiv 0\pmod{\Phi_n(q)}.
\end{align*}
\end{conjecture}

There are other classes of (super)congruences, in which truncated hypergeometric sums are compared with coefficients of modular forms.
One notable example, again from Van Hamme's list \cite[(M.2)]{Hamme}, is the supercongruence
$$
\sum_{k=0}^{p-1}\frac{(\frac12)_k^4}{k!^4}
\equiv\sum_{k=0}^{(p-1)/2}\frac{(\frac12)_k^4}{k!^4}
\equiv\gamma_p\pmod{p^3}
$$
for primes $p>2$, where the right-hand side represents the $p$-th coefficient in the $q$-expansion
$q\,(q^2;q^2)_\infty^4(q^4;q^4)_\infty^4=\sum_{n=1}^\infty\gamma_nq^n$ (of a modular form).
The supercongruence was settled by T.~Kilbourn \cite{Kilbourn} using $p$-adic methods.
An obstacle to producing a suitable $q$-analogue is related to the coefficients $\gamma_n$ (which already originate from a $q$-expansion!).
However, the machinery of hypergeometric motives, in particular, a method due to B.~Dwork,
allows one to reduce the proof of the (M.2) supercongruence to verifying the congruences
\begin{equation}
S(p^{s+1}-1)\equiv S(p^s-1)S(p-1)\pmod{p^3}
\label{eq:simple-congruence}
\end{equation}
for $s=1$ and $2$ (see \cite[Section~2.1]{LTYZ17}), where $S(N)$ denotes the truncation of the hypergeometric sum
$$
\sum_{n=0}^\infty\frac{(\frac12)_n^4}{n!^4}
$$
at the $N$-th place. So far we could not figure out a $q$-analogue of the `simpler' supercongruence \eqref{eq:simple-congruence},
though we expect that the method in this note is adaptable to these settings as well.

\medskip
\noindent
\textbf{Acknowledgements.}
The second author thanks Ofir Gorodetsky for a related chat on $q$-congruences.
The authors thank Mohamed El Bachraoui and the anonymous referees for their critical comments that helped to improve the exposition of the article.

\end{document}